\documentclass{amsart}
\usepackage{amsmath,amscd}
\usepackage{amsthm}
\usepackage{amssymb}
\usepackage{setspace}
\usepackage{fancyhdr}
\usepackage{stmaryrd}
\usepackage{bbm}
\input xy
\xyoption{all}
\usepackage[unicode]{hyperref}
\usepackage{color}

\newtheorem{theorem}{{Theorem}}[section]
\newtheorem{conjecture}[theorem]{{Conjecture}}

\newtheorem{corollary}[theorem]{{Corollary}}
\newtheorem{lemma}[theorem]{{Lemma}}
\newtheorem{proposition}[theorem]{Proposition}

\newtheorem{remark}[theorem]{{Remark}}

\theoremstyle{definition}

\newcommand{\QQ}{\mathbb{Q}}

\newcommand{\VV}{\mathbf{V}}
\newcommand{\Zp}{\mathbb{Z}_p}
\newcommand{\Qp}{\mathbb{Q}_p}
\newcommand{\Fp}{\mathbb{F}_p}

\newcommand{\FF}{\mathbb{F}}
\newcommand{\cR}{\mathcal{R}}
\newcommand{\cH}{\mathcal{H}}

\newcommand{\cZ}{\mathcal{Z}}

\newcommand{\cF}{\mathcal{F}}

\newcommand{\fa}{\mathfrak{a}}
\newcommand{\fp}{\mathfrak{p}}
\newcommand{\fq}{\mathfrak{q}}

\newcommand{\fm}{\mathfrak{m}}

\newcommand{\cO}{\mathcal{O}}

\newcommand{\ra}{\rightarrow}
\newcommand{\lra}{\longrightarrow}

\DeclareMathOperator{\Spec}{Spec}
\DeclareMathOperator{\Ker}{Ker}
\DeclareMathOperator{\Sp}{Sp}

\DeclareMathOperator{\End}{End}
\DeclareMathOperator{\Hom}{Hom}

\DeclareMathOperator{\Aut}{Aut}

\DeclareMathOperator{\Gal}{Gal}
\DeclareMathOperator{\GL}{GL}

\newcommand{\rInj}{\mathrm{Inj}}
\newcommand{\rsoc}{\mathrm{soc}}

\newcommand{\F}{\mathbb F}
\newcommand{\N}{\mathbb N}
\newcommand{\Q}{\mathbb Q}
\newcommand{\R}{\mathbb R}

\newcommand{\m}{\mathfrak{m}}
 
\newcommand{\Ht}{\mathrm{ht}} 
\newcommand{\cI}{\mathcal{I}} 
\newcommand{\ide}{\mathbbm{1}}

\newcommand{\bQp}{\overline{\Q}_p}

\newcommand{\Sym}{\mathrm{Sym}}

\newcommand{\Ind}{\mathrm {Ind}}

\providecommand{\cInd}{\mathrm{c}\textrm{-}\mathrm{Ind}}

\newcommand{\Ext}{\mathrm{Ext}}
\newcommand{\Ord}{\mathrm{Ord}}

\newcommand{\fn}{\mathfrak{n}}
\newcommand{\fC}{\mathfrak{C}}
\newcommand{\st}{\mathrm{st}}
\newcommand{\ps}{\mathrm{ps}}
\newcommand{\ver}{\mathrm{ver}}
\newcommand{\irr}{\mathrm{irr}}
\newcommand{\tr}{\mathrm{tr}}

\newcommand{\brho}{\overline{\rho}}
\newcommand{\rhobar}{\overline{\rho}}

\newcommand{\mSpec}{\mathrm{m}\textrm{-}\mathrm{Spec}}

\newcommand{\xto}[1][]{\xrightarrow{#1}}
\newcommand{\simto}{
\xto[\sim]} 

\newcommand{\matr}[4]{\begin{pmatrix}{#1}&{#2}\\ {#3}&{#4}\end{pmatrix}}
\newcommand{\smatr}[4]{\bigl(\begin{smallmatrix} {#1}& {#2}\\ {#3}&{#4}\end{smallmatrix}\bigl)}

\begin{document}

\title[The Breuil-M\'ezard conjecture for split residual representations]{The Breuil-M\'ezard conjecture for non-scalar split residual representations
}


\author{Yongquan Hu         \and
        Fucheng Tan 
}

\date{}
\maketitle

\begin{abstract}
We prove the Breuil-M\'{e}zard conjecture for split non-scalar residual representations of $\Gal(\bQp/\Q_p)$ by local methods. Combined with the cases previously proved in \cite{ki09} and \cite{pa12}, this completes the proof of the conjecture  (when $p\geq 5$). As a consequence, the local restriction in the proof of the Fontaine-Mazur conjecture in \cite{ki09} is removed. 
 
\end{abstract}

\tableofcontents

\section*{Notation}

\begin{itemize}

\item  $p\geq 5$ is a prime number. The $p$-adic valuation is normalized as $v_p(p)=1$.

\item $E/\Qp$ is a sufficiently large finite extension with ring of integers $\cO$, a (fixed) uniformizer $\varpi$, and residue field $\FF$. Its subring of Witt vectors is denoted by $W(\FF)$. 

\item For a number field $F$, the completion at a place $v$ is written as $F_v$, for which we fix a uniformizer denoted by $\varpi_v$.

\item For a local or global field $L$, $G_L=\Gal(\overline{L}/L)$. The inertia subgroup for the local field is written as $I_L$.

\item For each finite place $v$ in a number field $F$, fix a map $G_{F_v}\ra G_{F}$ by choosing an inclusion $\overline{F}\hookrightarrow \overline{F}_v$ of algebraic closures.

\item  $\epsilon: G_{\Qp}\ra \Zp^{\times}$ is the cyclotomic character, $\omega: G_{\Qp}\ra \Fp^{\times}$ is its reduction mod $p$, and $\tilde{\omega}$ is the Teichm\"{u}ler lifting of $\omega$. 

\item $\mathbbm{1}: G_{\Qp}\ra \Fp^{\times}$ is the trivial character. We also let $\mathbbm{1}$ denote other trivial representations, if no confusion arises.

\item Normalize the local class field map $\Qp^{\times}\ra G_{\Qp}^{\rm ab}$ so that uniformizers correspond to geometric Frobenii. Then a character of $G_{\Qp}$ will also be regarded as a character of $\Qp^{\times}$. 


\item For a ring $R$, $\mSpec R$ denotes the set of maximal ideals.

\item For $R$ a noetherian ring and $M$ a finite $R$-module of dimension at most $d$, let $\ell_{R_{\fp}}(M_{\fp})$ denote the length of the $R_{\fp}$-module $M_{\fp}$, and let $\cZ_d(M)=\sum_{\fp}\ell_{R_{\fp}}(M_{\fp})\fp$   for all $\fp\in \Spec R$ such that  $\mathrm{dim}R/\fp=d$. When the context is clear, we simply denote it by $\cZ(M)$.

\item  For $R$ a noetherian local ring with maximal ideal $\fm$ and $M$ a finite $R$-module,  and for an $\fm$-primary ideal $\fq$ of $R$, let $e_{\fq}(R,M)$ denote the Hilbert-Samuel multiplicity of $M$ with respect to $\fq$. 
We abbreviate $e_{\fm}(R,M)=e(R,M)$ and $e_{\fq}(R,R)=e_{\fq}(R)$. 

\item For $r\geq 0$, we let $\Sym^rE^2$ (resp. $\Sym^r\F^2$) be the usual symmetric power representation of $\GL_2(\Zp)$ (resp. of $\GL_2(\F_p)$, but viewed as a representation of $\GL_2(\Zp)$). 
\end{itemize}

\section{Introduction}

Consider the following data: 
\begin{enumerate}
\item[-] an integer $k\geq 2$,

\item[-] a representation $\tau: I_{\Qp}\ra \GL_2(E)$ with open kernel,

\item[-] a continuous character $\psi: G_{\Qp}\ra \cO^{\times}$ such that $\psi|_{I_{\Qp}}=\epsilon^{k-2}\mathrm{det} \tau$.
\end{enumerate}
We call such a triple $(k,\tau,\psi)$ a $p$-adic Hodge type. We say a $2$-dimensional continuous representation $\rho: G_{\Qp}\ra \GL_2(E)$ is of type $(k,\tau,\psi)$ if $\rho$ is potentially semi-stable (i.e. de Rham) such that its Hodge-Tate weights are $(0,k-1)$,  $\mathrm{WD}(\rho)|_{I_{\Qp}}\simeq \tau$, and $\mathrm{det}\rho\simeq \psi\epsilon$. Here $\mathrm{WD}(\rho)$ is the Weil-Deligne representation associated to $\rho$ by Fontaine \cite{Fo}.

By a result of Henniart \cite{He},  there is a unique finite dimensional smooth irreducible $\bQp$-representation $\sigma(\tau)$ (resp. $\sigma^{\rm cr}(\tau)$) of $\GL_2(\Zp)$ associated to $\tau$, such that for any infinite dimensional smooth absolutely irreducible representation $\pi$ of $\GL_2(\Qp)$ and the associated Weil-Deligne representation $LL(\pi)$ via classical local Langlands correspondence, we have $\Hom_{\GL_2(\Zp)}(\sigma(\tau),\pi)\neq 0$ if and only if $LL(\pi)|_{I_{\Qp}}\simeq \tau$ (resp. $\Hom_{\GL_2(\Zp)}(\sigma^{\rm cr}(\tau),\pi)\neq 0$ if and only if $LL(\pi)|_{I_{\Qp}}\simeq \tau$ and the monodromy operator is trivial). We remark that  $\sigma(\tau)$ and  $\sigma^{\rm cr}(\tau)$ differ only when $\tau=\chi\oplus\chi$ is scalar, in which case 
\[\sigma({\tau})=\tilde{\rm st}\otimes\chi\circ\det,\quad \sigma^{\rm cr}(\tau)=\chi\circ\det\]
where $\tilde{\st}$ is the inflation to $\GL_2(\Zp)$ of the Steinberg representation of $\GL_2(\F_p)$.

Enlarging $E$ if needed, we may and do assume $\sigma(\tau)$ is defined over $E$. Form the finite dimensional $\GL_2(\Zp)$-representation \[ \sigma(k,\tau)=\Sym^{k-2}E^2\otimes_E\sigma(\tau)\] and the semi-simplification $\overline{\sigma(k,\tau)}^{\rm ss}$ of the reduction modulo $\varpi$ of a $\GL_2(\Zp)$-stable $\cO$-lattice inside  $\sigma(k,\tau)$. Then $\overline{\sigma(k,\tau)}^{\rm ss}$ does not depend on the choice of the lattice. 

Recall that the finite dimensional irreducible $\FF$-representations of $\GL_2(\Zp)$  are of the form 
\[\sigma_{n,m}:=\Sym^n\FF^2\otimes\mathrm{det}^m, \quad n\in\{0,\cdots, p-1\}, m\in\{0,\cdots, p-2\}.\]
For each $\sigma_{n,m}$ let $a_{n,m}=a_{n,m}(k,\tau)$ be the multiplicity with which $\sigma_{n,m}$ occurs in $\overline{\sigma(k,\tau)}^{\rm ss}$. We have the obvious analogue in the crystalline case by considering 
\[\sigma^{\rm cr}(k,\tau):=\Sym^{k-2}E^2\otimes_E \sigma^{\rm cr}(\tau)\]
and denote the resulting numbers by $a_{n,m}^{\rm cr}=a_{n,m}^{\rm cr}(k,\tau)$.

Let $\overline{\rho}:G_{\Qp}\ra \GL_2(\F)$  be a continuous representation and $R^{\square}(\overline{\rho})$ be its universal framed deformation ring (\cite{ki08}). The following results on the structure of potentially semi-stable framed deformation rings are known.

\begin{theorem}[Kisin, \cite{ki08}]\label{pst}
There is a unique (possibly trivial) quotient $R^{\square,\psi}(k,\tau,\overline{\rho})$ (resp. $R_{\rm cr}^{\square,\psi}(k,\tau,\overline{\rho})$) of  $R^{\square}(\overline{\rho})$  such that 

(i) A map $x: R^{\square}(\overline{\rho})\ra E'$, for any finite extension $E'/E$, factors through  $R^{\square,\psi}(k,\tau,\overline{\rho})$ (resp. $R_{\rm cr}^{\square,\psi}(k,\tau,\overline{\rho})$) if and only if the Galois representation $\rho_{x}$ corresponding to $x$ is of type $(k,\tau,\psi)$ (resp. and is potentially crystalline).  

(ii) $R^{\square,\psi}(k,\tau,\overline{\rho})$ (resp. $R_{\rm cr}^{\square,\psi}(k,\tau,\overline{\rho})$) is $p$-torsion free.

(iii) $R^{\square,\psi}(k,\tau,\overline{\rho})[1/p]$  (resp. $R_{\rm cr}^{\square,\psi}(k,\tau,\overline{\rho})[1/p]$) is reduced, all of whose irreducible components are smooth of dimension $4$.
\end{theorem}

The following conjecture, the so-called Breuil-M\'ezard conjecture, relates the Hilbert-Samuel multiplicity of $R^{\square,\psi}(k,\tau,\brho)/\varpi$ (resp. $R^{\square,\psi}_{\rm cr}(k,\tau,\brho)/\varpi$) with the numbers $a_{n,m}$ (resp. $a_{n,m}^{\rm cr}$). 

\begin{conjecture}[Breuil-M\'{e}zard, \cite{bm1}]\label{bm} For any $(k,\tau,\psi)$ as above, we have
 
\begin{equation}\label{bme}e(R^{\square,\psi}(k,\tau,\overline{\rho})/\varpi)=\sum_{n,m}a_{n,m}(k,\tau)\mu_{n,m}(\overline{\rho}),\end{equation} 
\begin{equation}\label{bmecr}e(R_{\rm cr}^{\square,\psi}(k,\tau,\overline{\rho})/\varpi)=\sum_{n,m}a_{n,m}^{\rm cr}(k,\tau)\mu_{n,m}(\overline{\rho})\end{equation}for some integers  $\mu_{n,m}(\overline{\rho})$ which are independent of $k$, $\tau$ and $\psi$. 
\end{conjecture}

In particular, the conjecture implies that $$\mu_{n,m}(\overline{\rho})=e\left(R_{\rm cr}^{\square, \psi}(n+2,(\tilde{\omega}^{m})^{\oplus 2},\overline{\rho})/\varpi\right)$$ which can be computed. We refer the reader to \cite[1.1.6]{ki09} for these numbers, and remark that when $n=p-2$ and $\overline{\rho}$ is scalar, $\mu_{p-2,m}(\overline{\rho})=4$, as is shown in \cite{sa}. 

Conjecture \ref{bm} was proved by Kisin  \cite{ki09} in the cases that $\overline{\rho}$ is not (a twist of) an  extension of  $\mathbbm{1}$ by   $\omega$. He first proved the ``$\leq$'' part of (\ref{bme})  and (\ref{bmecr}) using the $p$-adic local Langlands \cite{co}, and then combined it with the (global) modularity lifting method to deduce the  ``$\geq$'' part. Years later, the conjecture was proved by Pa\v{s}k\={u}nas  \cite{pa12} for all  $\overline{\rho}$ with only scalar endomorphisms, using the $p$-adic local Langlands and his previous (local) results in \cite{pa10}. We prove, also using local methods (except for one global input due to Emerton \cite{em2}, see the introduction of \cite{pa12}), the following theorem (in the language of cycles of \cite{eg}), which in particular includes the remaining case of the conjecture (when $p\geq 5$).

\begin{theorem}[Remark \ref{bmgenericsplit}, Theorem \ref{bmng}, Theorem \ref{bmngcris}]\label{mainth}
For any continuous representation $\overline{\rho}: G_{\Qp}\ra \GL_2(\FF)$ which is isomorphic to the direct sum of two distinct characters, and for any $(k,\tau,\psi)$ as above, there are $4$-dimensional cycles $\cZ_{n,m}$ of $R^{\square}(\overline{\rho})$ which are independent of $(k,\tau,\psi)$ such that 
\[
\cZ(R^{\square,\psi}(k,\tau,\overline{\rho})/\varpi)=\sum_{n,m}a_{n,m}(k,\tau)\cZ_{n,m}.\]
\[
\cZ(R_{\rm cr}^{\square,\psi}(k,\tau,\overline{\rho})/\varpi)=\sum_{n,m}a_{n,m}^{\rm cr}(k,\tau)\cZ_{n,m}.\]
 Moreover, we have $\cZ_{n,m}=\cZ(R_{\rm cr}^{\square, \psi}(n+2,(\tilde{\omega}^{m})^{\oplus 2},\overline{\rho})/\varpi))$. In particular, the Breuil-M\'{e}zard Conjecture \ref{bm} is true. \end{theorem}

In fact, we prove Theorem \ref{mainth} in the language  of versal deformation rings $R^{\rm ver}(\rhobar)$ (see \S\ref{maps}). This implies the result,  as is explained in \S\ref{section6}.

Remark \ref{bmgenericsplit} is for the \emph{generic} case, i.e. for $\overline{\rho}=\chi_1\oplus\chi_2$ with $\chi_1\chi_2^{-1}\notin \{ \ide,\omega^{\pm}\}$, while Theorem \ref{bmng} and Theorem \ref{bmngcris} are for the \emph{non-generic} case, i.e. $\overline{\rho}\simeq \ide\oplus\omega$ (up to twist), which is a new result.

For the  proof, we follow closely that of \cite{pa12}, but have to deal with some extra complications, especially  when $\brho$ is a twist of $\ide\oplus \omega$, which we explain now. In \cite{pa12}, Pa\v{s}k\={u}nas developed  a general formalism to deduce the Breuil-M\'ezard conjecture, the key of which is to construct an appropriate  representation of $\GL_2(\Q_p)$ with coefficients in $R^{\rm ver}(\brho)$ satisfying several good properties, one of which is that it gives the universal deformation of $\overline{\rho}$ over $R^{\ver}(\brho)$ via Colmez's functor (in fact, to do so, we should work with deformation rings with fixed determinant, but we ignore this issue in this introduction).  Then, using the $p$-adic local Langlands, he reduces the proof of the conjecture to  representation theory of $\GL_2(\Q_p)$. 
When $\brho$ is split and generic, such a construction can be done easily and essentially follows from that of \cite{pa12}. 

However, we are not able to do it directly when $\brho$ is a twist of $\ide\oplus\omega$. In contrast, such a $\GL_2(\Q_p)$-representation over the pseudo-deformation ring of (the trace of) $\brho$  is known, thanks to Pa\v{s}k\={u}nas' previous work \cite{pa10}.
This naturally suggests that we first mimic Pa\v{s}k\=unas' strategy in the setting of potentially semi-stable pseudo-deformation rings, and then pass to the corresponding versal deformation rings, as   Kisin did in \cite{ki09}.  There are however two complications in doing so. The first one is that the $\GL_2(\Q_p)$-representation over the pseudo-deformation ring constructed in \cite{pa10} is not flat, which makes the arguments more involved when verifying the setting of \cite{pa12}; see \S\ref{subsection-M_2}. The second is that even if the (analogous) conjecture for pseudo-deformation rings is proven,  the local argument in \cite[\S1.7]{ki09} only gives the inequality ``$\leq$''. To resolve these, we construct and  study morphisms among various deformation rings, and reduce the conjecture to the (analogous) statement for pseudo-deformation rings and to the cases which have been treated in \cite{pa12}.  Thus, our proof may also be viewed as a refinement of the local argument in \cite{ki09}.  
 \bigskip

With the main result of \cite{pa12}  and Theorem \ref{mainth} in hand, Kisin's original proof \cite{ki09} applies to give the Fontaine-Mazur conjecture for geometric Galois representations $\rho:G_{\Q}\ra \GL_2(\cO)$ such that $\overline{\rho}|_{G_{\Qp}}$ is  a twist of an extension of $\ide$ by $\omega$, split or not. These are complementary to the  cases treated in \cite{ki09}. Putting them together, we have
the following theorem (recall that $p\geq 5$).

 \begin{theorem}\label{fm}
 
 Let $\rho: G_{\Q}\ra \GL_2(\cO)$ be a continuous representation which is unramified away from a finite set of primes,  whose residual representation $\rhobar$ is odd with restriction $\overline{\rho}|_{\Q(\zeta_p)}$ being absolutely irreducible. If $\rho|_{G_{\Qp}}$ is potentially semi-stable with distinct Hodge-Tate weights, then $\rho$ comes from a modular form, up to a twist.  
 \end{theorem}

 Note that the majority cases of Theorem \ref{fm} was also  proved by Emerton \cite{em2}, namely the cases for which  $\overline{\rho}|_{G_{\Qp}}$ is  not a twist of an extension   of $\omega$ by  $\ide$ or an extension of  $\ide$ by  $\ide$. Thus the only new case proved here is when $\overline{\rho}|_{G_{\Qp}}$ is a twist of the direct sum $\ide\oplus \omega$. \bigskip

The paper is organized as follows. Section 2 and Section 4 are devoted to the study of the pseudo-deformation rings using representation theory of $\GL_2(\Q_p)$ via the theory developed in \cite{pa10},\cite{pa12}. In Section 3, we give explicit descriptions of certain deformation rings and maps among them.  We prove Theorem \ref{mainth}  in Section 5 and prove  Theorem \ref{fm} in Section 6.
 \bigskip

 \textbf{Acknowledgments.}
The authors are deeply indebted to Mark Kisin and Vytautas Pa\v{s}k\={u}nas for the works \cite{ki09} and  \cite{pa12}. The first named author would like to thank Xavier Caruso and Laurent Moret-Bailly for several discussions, and Vytautas Pa\v{s}k\={u}nas for helpful correspondences. They thank the referee for providing many constructive  comments and for help in  improving the content of this paper.
 They are grateful to the Morningside Center of Mathematics and the Max-Planck Institute for Mathematics for their hospitality in the final stages of the project.

\section{Preparations on $\F$-representations of $\GL_2(\Q_p)$} 

In  this section, we redefine and study Kisin's map $\theta$ \cite[1.5.11]{ki09}. It will be used in \S 4.\medskip

Let $G:=\GL_2(\Qp)$, $K:=\GL_2(\Zp)$ and $Z\subset G$ be the centre. Denote by  $P\subset G$   the upper triangular Borel subgroup, by $I\subset K$ the upper triangular Iwahori subgroup, and by $I_1\subset K$ the  upper triangular pro-$p$-Iwahori subgroup.

Let $\mathrm{Mod}_{G}^{\rm sm}(\cO)$ be the category of smooth $G$-representations on $\cO$-torsion modules and $\mathrm{Mod}_{G}^{\rm l,fin}(\cO)$ be its full subcategory consisting of locally finite objects. Here an object $\tau\in\mathrm{Mod}_{G}^{\rm sm}(\cO)$ is said to be \emph{locally finite} 
 if for all $v\in \tau$ the $\cO[G]$-submodule generated by $v$ is of finite length. For $\tau\in \mathrm{Mod}_G^{\rm l,fin}(\cO)$, we write $\rsoc_G\tau$ for its $G$-socle, namely the largest semi-simple sub-representation of $\tau$. Let $\mathrm{Mod}_{G}^{\rm sm}(\F)$ and $\mathrm{Mod}_{G}^{\rm l,fin}(\F)$ be respectively  the full subcategory consisting of $G$-representations on $\F$-modules, i.e. killed by $\varpi$. Moreover, for a continuous character $\zeta:Z\ra\cO^{\times}$, adding the subscript $\zeta$ in any of the above categories indicates  the corresponding full subcategory of $G$-representations with central character $\zeta$.   
 
Let $\mathrm{Mod}_{G}^{\rm pro}(\cO)$ be the category of compact $\cO\llbracket K\rrbracket$-modules with an action of $\cO[G]$ such that the two actions coincide when restricted to $\cO[K]$. This category is anti-equivalent to $\mathrm{Mod}_{G}^{\rm sm}(\cO)$   
under the Pontryagin dual $\tau\mapsto\tau^{\vee}:=\Hom_{\cO}(\tau,E/\cO)$, the latter being equipped with the compact-open topology.  Finally let $\mathfrak{C}_{\zeta}(\cO)$ and $\mathfrak{C}_{\zeta}(\F)$ be respectively the full subcategory of $\mathrm{Mod}_G^{\rm pro}(\cO)$ anti-equivalent to  $\mathrm{Mod}_{G,\zeta}^{\rm l,fin}(\cO)$ and $\mathrm{Mod}_{G,\zeta}^{\rm l,fin}(\F)$.

\subsection{Some $\F$-representations of $G$}\label{subsection-F-rep}

Fix an integer $r\in \{0,...,p-1\}$ and consider the representation $\Sym^{r}\F^2$ of $KZ$ obtained by letting $p\in Z$ act trivially.  Fix a continuous character $\chi:\Q_p^{\times}\ra \F^{\times}$  and  $\lambda \in \F$. For our purpose we will assume:\medskip
 
({\bf H})  $\lambda\neq0 $ and $(r,\lambda)\neq (p-1,\pm 1)$.\medskip

Write $I({\Sym^{r}\F^2}):=\cInd_{KZ}^G\Sym^{r}\F^2$, the compact induction of $\Sym^{r}\F^2$ from $KZ$ to $G$, and $I_{\chi}(\Sym^r\F^2):=I(\Sym^r\F^2)\otimes\chi\circ\det$. By \cite[Proposition 8]{BL}, we have 
$\End_{G}(I_{\chi}(\Sym^r\F^2))\cong \F[T_r]$
for certain Hecke operator  $T_{r}$ (as normalized in \cite[\S3.1]{BL} or in \cite[1.2.1]{ki09}). We will often write $T=T_{r}$ if no confusion is caused.  Write \[\pi(r,\lambda,\chi):=I_{\chi}(\Sym^r\F^2)/(T-\lambda).\] By \cite[Theorem 30]{BL}, $\pi(r,\lambda,\chi)$ is an irreducible principal series if $(r,\lambda)\neq (0,\pm1)$ (under our assumption (\textbf{H})), and is reducible of length 2 if $(r,\lambda)=(0,\pm1)$ in which case  we have a \emph{non-split} short exact sequence:
\[0\ra \Sp\otimes\chi\mu_{\pm1}\circ\det\ra \pi(0,\pm1,\chi)\ra \chi\mu_{\pm1}\circ\det\ra0\]
where $\Sp$ denotes the Steinberg representation of $G$ and $\mu_{\pm1}:\Q_{p}^{\times}\ra \F^{\times}$ denotes the unramified character  sending $p$ to $\pm1$.

Since $\F[T]$  acts freely on $I_{\chi}(\Sym^r\F^2)$ by \cite[Theorem 19]{BL}, for each $n\in \N$  we have a natural $G$-equivariant injection
\[(T-\lambda): I_{\chi}(\Sym^r\F^2)/(T-\lambda)^n\ra I_{\chi}(\Sym^r\F^2)/(T-\lambda)^{n+1}.\]
Write $\pi_n(r,\lambda,\chi):=I_{\chi}(\Sym^r\F^2)/(T-\lambda)^n$ for $n\geq 1$ so that $\pi_1(r,\lambda,\chi)=\pi(r,\lambda,\chi)$. For convenience, we set $\pi_0(r,\lambda,\chi):=0$.
Then, for $1\leq m\leq n$, we have an exact sequence of $G$-representations:
\begin{equation}\label{equation-pi-m-n}0\ra \pi_{m}(r,\lambda,\chi)\overset{(T-\lambda)^{n-m}}{\lra} \pi_{n}(r,\lambda,\chi)\ra \pi_{n-m}(r,\lambda,\chi)\ra0\end{equation}
which is non-split because $\F[T]$ acts freely on $I_{\chi}(\Sym^r\F^2)$. 

Put   $$\pi_{\infty}(r,\lambda,\chi):=\varinjlim_n\pi_n(r,\lambda,\chi).$$ Then $\pi_{\infty}(r,\lambda,\chi)$ is a smooth locally finite $\F$-representation of $G$ with central character $\chi^2\omega^r$. 
Taking $m=1$ and passing to the limit over $n$ in (\ref{equation-pi-m-n}),  we obtain a non-split exact sequence
\begin{equation}\label{equation-pi-infty}
0\ra \pi(r,\lambda,\chi)\ra \pi_{\infty}(r,\lambda,\chi)\ra \pi_{\infty}(r,\lambda,\chi)\ra0.\end{equation}
 
\begin{lemma}\label{Lemma-soc-pi-infty}
(i) The $\F$-vector space $\Hom_{G}\left(\pi(r,\lambda,\chi),\pi_{\infty}(r,\lambda,\chi)\right) $ is of dimension 1  and is spanned by the second  arrow constructed in (\ref{equation-pi-infty}). In particular, any non-zero $G$-equivariant morphism $\pi(r,\lambda,\chi)\ra\pi_{\infty}(r,\lambda,\chi)$ is injective.

(ii)  We have  \[\rsoc_G\pi_{\infty}(r,\lambda,\chi)=\rsoc_{G}\pi(r,\lambda,\chi)=\left\{\begin{array}{lll}\pi(r,\lambda,\chi)&\mathrm{if}\ (r,\lambda)\neq (0,\pm1)\\
 \Sp\otimes\chi\mu_{\pm1}\circ\det&\mathrm{if} \ (r,\lambda)=(0,\pm1).\end{array}\right. \] 
\end{lemma}
\begin{proof}
We give a  proof for the sake of completeness although the argument is standard. To simplify the notation, we write   $\pi_n$ for $\pi_n(r,\lambda,\chi)$ (where $n\in\N\cup\{\infty\}$).

(i) By construction it suffices to prove  that for any $n\geq 1$, the $\F$-vector space  $\Hom_G(\pi_1,\pi_n)$ is of dimension 1 and is  spanned by $(T-\lambda)^{n-1}:\pi_1\ra \pi_n$.
 This is clear  when $n=1$.  Let $n\geq 2$ and assume the assertion is true for $n-1$. Then the exact sequence (\ref{equation-pi-m-n}) with $m=n-1$  induces  
 \[0\ra \Hom_G(\pi_1,\pi_{n-1})\ra \Hom_{G}(\pi_1,\pi_n)\ra \Hom_G(\pi_1,\pi_1).\]
We deduce that $\Hom_G(\pi_1,\pi_n)$ is of dimension $\leq 2$, and the equality holds if and only if the last arrow is surjective if and only if (\ref{equation-pi-m-n}) is split (when $m=n-1$). Since (\ref{equation-pi-m-n}) is non-split, the result follows.   

(ii) The second equality is clear by what we have recalled. For the first one,  if $(r,\lambda)\neq (0,\pm1)$, then   $\pi_1$ is irreducible and each irreducible constituent of $\pi_{\infty}$ is isomorphic to $\pi_1$ so the lemma follows from (i).  

 
Assume   now  $(r,\lambda)=(0,\pm1)$ so that $\rsoc_G\pi_1=\Sp\otimes\chi\mu_{\pm1}\circ\det$.  We assume moreover $\lambda=1$ and $\chi$ is trivial;  the general  case  can be deduced by twisting. In particular, the central character of $\pi_n$ is trivial. Clearly if $\pi$ is an irreducible smooth $\F$-representation of $G$ such that  $\Hom_G(\pi,\pi_{\infty})\neq 0$ then $\pi\cong \Sp$ or $\pi\cong \ide$ (the trivial $\F$-representation of $G$). Moreover, by (i) the natural morphism $\Hom_G(\pi_1,\pi_{\infty})\ra\Hom_G(\Sp,\pi_{\infty})$ is non-zero,  hence $\dim_{\F}\Hom_{G}(\Sp,\pi_{\infty})\geq 1$ and    $\Hom_G(\ide,\pi_{\infty})=0$. 

We are left to show  $\dim_{\F}\Hom_G(\Sp,\pi_{\infty})=1$, or equivalently $\dim_{\F}\Hom_G(\Sp,\pi_{n})=1$ for all $n\geq 1$.   
For each $n\geq 2$ we define $\tau_{n}$ to be the kernel of the  composition  $\pi_n\twoheadrightarrow \pi_1\twoheadrightarrow \ide$. Then $\tau_n$ fits into the exact sequence
\begin{equation}\label{equation-tau-n}0\ra \pi_{n-1}\ra \tau_{n}\ra \Sp\ra0.\end{equation}  If we had $\dim_{\F}\Hom_G(\Sp,\pi_k)\geq 2$ for some $k\in\N$ which we choose to be the smallest,  then $k\geq 2$ and the sequence (\ref{equation-tau-n}) with $n=k$ must split and would induce an exact sequence
\[0\ra \Sp\oplus \pi_{k-1}\ra \pi_k\ra \ide\ra0.\]  
Since $\Hom_G(\pi_{k-1},\ide)\neq 0$ and $\Ext^1_{G/Z}(\ide,\ide)=0$ (since $p\neq 2$, see \cite[\S10.1]{pa10}), this would imply $\dim_{\F}\Hom_G(\pi_k,\ide)\geq 2$ hence
\[\dim_{\F}\Hom_{K}(\Sym^{0}\F^2,\ide)=\dim_{\F}\Hom_{G}(I(\Sym^{0}\F^2),\ide)\geq 2.\]
This being  impossible,  the assertion follows. 
\end{proof}

Let $\rInj_G\pi(r,\lambda,\chi)$ be an injective envelope of $\pi(r,\lambda,\chi)$ in $\mathrm{Mod}_{G,\zeta}^{\rm l,fin}(\F)$, where $\zeta:Z\ra\cO^{\times}$ is a  continuous character whose reduction modulo $\varpi$ is equal to $\chi^2\omega^r$, the central character of $\pi(r,\lambda,\chi)$. Lemma  \ref{Lemma-soc-pi-infty} implies the existence of a $G$-equivariant injection \[\theta: \pi_{\infty}(r,\lambda,\chi)\hookrightarrow \rInj_G\pi(r,\lambda,\chi).\] Such an injection need  not be unique. We will show later that the image of $\theta$ does not depend on the choice; see Corollary \ref{Coro-Imaga-theta}.

Let $\cH$ be the Hecke algebra associated to $\cInd_{I_1Z}^G\zeta$ and $\mathrm{Mod}_{\cH}$ the category of $\cH$-modules. Denote by $\cI:\mathrm{Mod}_{G,\zeta}^{\rm sm}(\F)\ra \mathrm{Mod}_{\cH}$  the left exact functor induced by taking $I_1$-invariants and $\R^i\cI$ its right derived functors for $i\geq 1$, cf. \cite[\S5.4]{pa10} for a collection of properties about this functor.  Recall  the following result.

\begin{lemma}\label{Lemma-Ext^2(H)}
Let $\pi$ be a smooth irreducible non-supersingular $\F$-representation of $G$. Then 
\begin{enumerate}
\item[(i)] $\Ext^2_{\cH}(\cI(\pi),*)=0$;
\item[(ii)]  $\Ext^1_{\cH}(\cI(\pi),\cI(\pi(r,\lambda,\chi)))=0$ except when $\pi\cong \rsoc_G\pi(r,\lambda,\chi)$ in which case the space is of dimension  $1$ over $\F$.
\end{enumerate}
\end{lemma}
\begin{proof}
(i) It is a special case of   \cite[Lemma 5.24]{pa10}. 
(ii) If $(r,\lambda)\neq (0,\pm1)$ so that $\pi(r,\lambda,\chi)$ is irreducible, it is a special case of \cite[Lemma 5.27(ii)]{pa10}. If $(r,\lambda)=(0,\pm1)$,  then it follows from \cite[Lemma 5.27(iii)]{pa10}, using (i) for the second assertion.
\end{proof}
\begin{proposition}\label{Prop-Image-theta}
The morphism $\theta$ identifies $\pi_{\infty}(r,\lambda,\chi)$ with the largest $G$-stable subspace of $\rInj_G\pi(r,\lambda,\chi)$  generated by its $I_1$-invariants. In other words, $\theta$ induces an isomorphism
\[\theta: \pi_{\infty}(r,\lambda,\chi)\simto \langle G\cdot (\rInj_G\pi(r,\lambda,\chi))^{I_1}\rangle. \] 
\end{proposition}
\begin{proof}
To simplify the notation, we  write $\pi_{n}$ for $\pi_n(r,\lambda,\chi)$ where $n\in\N\cup\{\infty\}$.

Let $\pi$ be an irreducible object in $\mathrm{Mod}_{G,\zeta}^{\rm sm}(\F)$. Recall that we have the following exact sequence 
\[0\ra \Ext^1_{\cH}(\cI(\pi),\cI(\pi_{\infty}))\overset{\mathcal{T}}{\ra} \Ext^1_{G,\zeta}(\pi,\pi_{\infty})\ra \Hom_{\cH}(\cI(\pi),\R^1\cI(\pi_{\infty}))\]
see for example \cite[\S 5.4]{pa10}, where $\mathcal{T}:\mathrm{Mod}_{\cH}\ra \mathrm{Mod}_{G,\zeta}^{\mathrm{sm}}(\F)$ denotes the functor $M\mapsto  M\otimes_{\cH}\cInd_{I_1Z}^G\zeta$ and $\Ext^1_{G,\zeta}$ indicates that the extensions are calculated in the category $\mathrm{Mod}^{\rm sm}_{G,\zeta}(\F)$. 
By the main result of \cite{Ol09}, an extension $0\ra \pi_{\infty}\ra V\ra \pi\ra0$ lies in the image of $\mathcal{T}$ if and only if $V$ is generated by its $I_1$-invariants, i.e. $V=\langle G\cdot V^{I_1}\rangle$. We will show  $\Ext^1_{\cH}(\cI(\pi),\cI(\pi_{\infty}))=0$  which will imply the assertion.  

By definition of $\pi_{\infty}$, we have an isomorphism $\cI(\pi_{\infty})\cong \varinjlim_n\cI(\pi_n)$ as $\cH$-modules which induces $\Ext^1_{\cH}(\cI(\pi),\cI(\pi_{\infty}))\cong \varinjlim_n\Ext^1_{\cH}(\cI(\pi),\cI(\pi_n))$. The latter isomorphism holds because $\cI(\pi)$ is a finitely presented $\cH$-module, see \cite{Vi04}. So it suffices to show that the transition map
\[\alpha_n: \Ext^1_{\cH}(\cI(\pi),\cI(\pi_{n})) {\ra} \Ext^1_{\cH}(\cI(\pi),\cI(\pi_{n+1}))\] 
is zero for any $n\geq 1$. 
By Lemma \ref{Lemma-Ext^2(H)}, we may assume $\pi= \rsoc_G\pi_1$. The exact sequence (\ref{equation-pi-m-n})  induces a sequence of $\cH$-modules
\begin{equation}\label{equation-I(pi_n)}0\ra \cI(\pi_{n})\ra \cI(\pi_{n+1})\ra\cI(\pi_1)\ra0,\end{equation}
which is still exact by the main result of \cite{Ol09} because $\pi_{n+1}$ is generated by its $I_1$-invariants.  Applying $\Hom_{\cH}(\cI(\pi),*)$ to it  and using Lemma \ref{Lemma-Ext^2(H)}(i) and the fact that $\Hom_{\cH}(\cI(\pi),\cI(\pi_n))\cong \Hom_G(\pi,\pi_n)\cong \F$ for all $n\geq 1$ by Lemma \ref{Lemma-soc-pi-infty}(ii),  we get a long exact sequence
\begin{multline*}0\ra \Hom_{\cH}(\cI(\pi),\cI(\pi_1))\ra \Ext^1_{\cH}(\cI(\pi),\cI(\pi_{n}))\overset{\alpha_n}{\ra} \Ext^1_{\cH}(\cI(\pi),\cI(\pi_{n+1}))\\
\ra\Ext^1_{\cH}(\cI(\pi),\cI(\pi_1))\ra0.\end{multline*}
 Since this holds for all $n\geq 1$, an induction on $n$, using Lemma \ref{Lemma-Ext^2(H)}(ii), implies that all dimensions over $\F$ appeared in the last exact sequence are equal to 1, and the morphism $\alpha_n$
 must be zero. This finishes the proof.
\end{proof}

\begin{corollary}\label{Coro-Imaga-theta}
The image of $\theta$ does not depend on the choice of $\theta$. More generally, for any non-zero morphism $\theta':\pi_{\infty}(r,\lambda,\chi)\ra \rInj_G\pi(r,\lambda,\chi)$, its image coincides with that of $\theta$.
\end{corollary}
\begin{proof}
The first assertion follows from Proposition \ref{Prop-Image-theta}.  Since $\theta'$ is nonzero, we can define the largest integer $k\in\N$ such that $\theta'$ factors through $\pi_{k}(r,\lambda,\chi)$.  Then the induced map $$\pi_{\infty}(r,\lambda,\chi)/\pi_k(r,\lambda,\chi)\ra \rInj_G\pi(r,\lambda,\chi)$$ must be an injection, using  Lemma \ref{Lemma-soc-pi-infty} when $\pi(r,\lambda,\chi)$ is reducible.  The  quotient $\pi_{\infty}(r,\lambda,\chi)/\pi_k(r,\lambda,\chi)$ is isomorphic to $\pi_{\infty}(r,\lambda,\chi)$ by (\ref{equation-pi-infty}), so we can apply the first assertion to conclude.  
\end{proof}

\begin{corollary}\label{Coro-Hom-pi-infty}
For any smooth irreducible  $\F$-representation $\sigma$ of $K$, $\theta$ induces an isomorphism
\[\Hom_{K}(\sigma,\pi_{\infty}(r,\lambda,\chi))\cong \Hom_{K}(\sigma,\rInj_G\pi(r,\lambda,\chi)).\]
Moreover, the two spaces are non-zero if and only if $\Hom_K(\sigma,\pi(r,\lambda,\chi))\neq 0$.
\end{corollary}
\begin{proof}
The second assertion follows from the first one by definition of $\pi_{\infty}(r,\lambda,\chi)$. By Frobenius reciprocity, we need to show that the injection (induced from $\theta$)
\[\Hom_G(I_{\chi'}(\sigma),\pi_{\infty}(r,\lambda,\chi))\hookrightarrow \Hom_G(I_{\chi'}(\sigma),\rInj_{G}\pi(r,\lambda,\chi))\]
is an isomorphism, where $\chi':\Q_p^{\times}\ra \F^{\times}$ is the character making the central character of $I_{\chi'}(\sigma)$ to be that of $\pi(r,\lambda,\chi)$.  But this follows from Proposition \ref{Prop-Image-theta} since the image of   $I_{\chi'}(\sigma)\ra \rInj_G\pi(r,\lambda,\chi)$ is generated by its $I_1$-invariants, hence lies in $\theta(\pi_{\infty}(r,\lambda,\chi))$.  
\end{proof}

\begin{remark}
The above results (Proposition \ref{Prop-Image-theta} and Corollaries \ref{Coro-Imaga-theta}, \ref{Coro-Hom-pi-infty}) hold true in the case $(r,\lambda)=(p-1,\pm1)$. To see this one can either modify the above proofs or apply (the proof of) \cite[1.5.5]{ki09}.
\end{remark}

The next lemma will be used in the proof of Proposition \ref{Prop-J-prime}.
\begin{lemma}\label{Lemma-M(sigma)}
For any smooth irreducible $\F$-representation $\sigma$ of $K$, the following sequence induced by (\ref{equation-pi-infty}) is exact
\[0\ra \Hom_K(\sigma,\pi(r,\lambda,\chi))\ra\Hom_K(\sigma,\pi_{\infty}(r,\lambda,\chi))\ra\Hom_K(\sigma,\pi_{\infty}(r,\lambda,\chi))\ra0.\]
\end{lemma}
\begin{proof}
To simplify the notation, we write   $\pi_n$ for $\pi_n(r,\lambda,\chi)$ (where $n\in\N\cup\{\infty\}$). We may assume $\chi$ is trivial by twisting. We also assume that $\Hom_K(\sigma,\pi_1)\neq0$, otherwise the assertion is trivial by Corollary \ref{Coro-Hom-pi-infty}.  By \cite[Theorem 34]{BL}, this implies  that $\sigma\cong\Sym^{r}\F^2$ if $r\notin\{0,p-1\}$, and $\sigma\in\{\Sym^0\F^2,\Sym^{p-1}\F^2\}$ otherwise. Moreover, in all cases, we have $\dim_{\F}\Hom_K(\sigma,\pi_1)=1$. 

Since $\Hom_K(\sigma,\pi_{\infty})\cong\varinjlim_{n\geq1}\Hom_K(\sigma,\pi_n)$, it suffices to prove the exactness of the  sequence 
\[0\ra\Hom_{K}(\sigma,\pi_1)\ra\Hom_K(\sigma,\pi_n)\ra\Hom_K(\sigma,\pi_{n-1})\ra0\]
for all $n\geq 1$, or equivalently, to  prove $\dim_{\F}\Hom_K(\sigma,\pi_n)=n$ for all $n\geq 1$. 
This is true if $\sigma\cong\Sym^{r}\F^2$, since an easy induction on $n$ shows that 
\[\Hom_K(\Sym^r\F^2,\pi_n)\cong \Hom_G(I(\Sym^r\F^2),\pi_n) \]
is of dimension $n$ over $\F$, with a basis given by 
\[\bigl\{ I(\Sym^{r}\F^2)\twoheadrightarrow \pi_i\hookrightarrow \pi_n, 1\leq i\leq n\bigr\}, \]
where the first arrow is the natural quotient map and the second is given by (\ref{equation-pi-m-n}).   

If   $\sigma\ncong\Sym^r\F^2$, then we have $r\in\{0,p-1\}$ and $\sigma\cong\Sym^{p-1-r}\F^2$ so that 
\[\rsoc_K(\pi_n)=(\Sym^{r}\F^2)^{\oplus n_1}\oplus(\Sym^{p-1-r}\F^2)^{\oplus n_2}\]
for some $n_1,n_2\geq 0$.    By the case already treated, we have $n_1=n$. On the other hand, since $\dim_{\F}\cI(\pi_1)=2$, an induction on $n$ using the exact sequence (\ref{equation-I(pi_n)}) shows that $\dim_{\F}\cI(\pi_n)=2n$. Using the fact that $\Ind_I^K\ide\cong \Sym^0\F^2\oplus\Sym^{p-1}\F^2$, we show that $\rsoc_K\pi_n$ is generated by $\pi_n^{I_1}$ as a $K$-representation, so that 
\[2n=\dim_{\F}\cI(\pi_n)=\dim_{\F}(\rsoc_K\pi_n)^{I_1}=n_1+n_2.\]
This implies $n_2=n$ and finishes the proof.
\end{proof}
\subsection{The prime ideal $J$}
\label{subsection-J}

We keep the notation   in the preceding subsection. Let $\pi^{\vee}(r,\lambda,\chi)$ be the Pontryagin dual of $\pi(r,\lambda,\chi)$ and $\pi^{\vee}_{\infty}(r,\lambda,\chi)$ be that of $\pi_{\infty}(r,\lambda,\chi)$. They are objects in $\mathfrak{C}_{\zeta}(\F)$, the dual category of $\mathrm{Mod}^{\rm l,fin}_{G,\zeta}(\F)$. Dualizing the sequence (\ref{equation-pi-infty}), we get an injective $G$-equivariant endomorphism of $\pi_{\infty}^{\vee}(r,\lambda,\chi)$ which we denote by $S$. We have 
\begin{equation}\label{equation-sequence-S}
0\ra \pi_{\infty}^{\vee}(r,\lambda,\chi)\overset{S}{\ra} \pi_{\infty}^{\vee}(r,\lambda,\chi)\ra\pi^{\vee}(r,\lambda,\chi)\ra0\end{equation} and
$\pi^{\vee}_{\infty}(r,\lambda,\chi)\cong \varprojlim_n\pi_{\infty}^{\vee}(r,\lambda,\chi)/S^{n}$
so that $\pi_{\infty}^{\vee}(r,\lambda,\chi)$ can be naturally viewed as an $\F\llbracket S\rrbracket$-module.

Let $\widetilde{P}:=\mathrm{Proj}_{\mathfrak{C}_{\zeta}(\cO)}\pi^{\vee}(r,\lambda,\chi)$, a projective envelope of $\pi^{\vee}(r,\lambda,\chi)$ in $\mathfrak{C}_{\zeta}(\cO)$, and $\widetilde{E}:=\End_{\mathfrak{C}_{\zeta}(\cO)}(\widetilde{P})$ which acts naturally on $\widetilde{P}$.  Then we have an isomorphism $\widetilde{P}\otimes_{\cO}\F\cong \big(\rInj_G\pi(r,\lambda,\chi)\big)^{\vee}$ in  $\mathfrak{C}_{\zeta}(\F)$. 
The injection  $\theta: \pi_{\infty}(r,\lambda,\chi) \hookrightarrow \rInj_{G}\pi(r,\lambda,\chi)$ chosen in \S\ref{subsection-F-rep} induces a surjection in $\mathfrak{C}_{\zeta}(\cO)$:
\[\theta^{\vee}:\widetilde{P}\twoheadrightarrow \widetilde{P}\otimes_{\cO}\F \twoheadrightarrow \pi_{\infty}^{\vee}(r,\lambda,\chi).\]
Define a right ideal of $\widetilde{E}$ as follows:
\[J:=\{\varphi\in \widetilde{E}: \theta^{\vee}\circ \varphi=0\}.\]
According to Proposition \ref{Prop-Image-theta}, $J$ does not depend on the choice of $\theta$.

\begin{remark}
Note that $\widetilde{E}$ need not be commutative, see \cite[\S9]{pa10}. In fact, it is shown in \cite{pa10} that $\widetilde{E}$ is commutative if and only if $(r,\lambda)\neq(p-2,\pm1)$.
\end{remark}
 
Let $W$ be a smooth   $\F$-representation of $K$ of finite length. Recall from \cite[Definition 2.2]{pa12} the compact left $\widetilde{E}$-module $M(W)$ defined as
\[M(W):=\Hom_{\cO\llbracket K\rrbracket}^{\rm cont}(\widetilde{P},W^{\vee})^{\vee}.\]
The main result of \cite{EP} implies that $\widetilde{P}$ is also projective in $\mathrm{Mod}_{K,\zeta}^{\rm pro}(\cO)$, so that $M(\cdot)$ is an exact functor.  
Write $\mathrm{Ann}(M(W))$ for the annihilator of $M(W)$ in $\widetilde{E}$, i.e.  $$\mathrm{Ann}(M(W)):=\{\varphi\in \widetilde{E}: u\circ\varphi=0, \ \forall u\in \Hom_K(\widetilde{P},W^{\vee})\}.$$

\begin{proposition}\label{Prop-J-prime}
Let $\sigma$ be a smooth  irreducible $\F$-representation of $K$.  We have $M(\sigma)\neq 0$ if and only if $\Hom_{K}(\sigma,\pi(r,\lambda,\chi))\neq 0$. If this is the case, then 
 $J=\mathrm{Ann}(M(\sigma))$. Moreover,  $\widetilde{E}/J \cong \F\llbracket S\rrbracket$, and $J$ is a (two-sided) prime  ideal of $\widetilde{E}$.
\end{proposition}
\begin{proof}
We write $\pi_n=\pi_n(r,\lambda,\chi)$ for all $n\in\N\cup\{\infty\}$ to simplify the notation.

The first assertion follows from Corollary \ref{Coro-Hom-pi-infty} and  that $\widetilde{P}\otimes_{\cO}\F\cong \big(\rInj_G\pi(r,\lambda,\chi)\big)^{\vee}$.
Assume $\Hom_{K}(\sigma,\pi_1)\neq0$.   Dualizing, Corollary \ref{Coro-Hom-pi-infty} gives an isomorphism  
\begin{equation}\label{equation-J-prime}\Hom_{\cO\llbracket K\rrbracket}^{\rm cont}(\pi_{\infty}^{\vee},\sigma^{\vee})\simto\Hom_{\cO\llbracket K\rrbracket}^{\rm cont}(\widetilde{P},\sigma^{\vee})\end{equation}
so that $J\subseteq \mathrm{Ann}(M(\sigma))$. Conversely, let $\varphi\in \mathrm{Ann}(M(\sigma))$ and assume $\theta^{\vee}\circ \varphi\neq 0$. The image of $\theta^{\vee}\circ\varphi$ is then a non-zero sub-object of $\pi_{\infty}^{\vee}$ whose Pontryagin dual is  the image of some  non-zero morphism $\theta':\pi_{\infty}\ra \rInj_G\pi_1$. However, we have $\Hom_K(\sigma,\mathrm{Im}(\theta'))\neq 0$ by Corollaries \ref{Coro-Imaga-theta} and \ref{Coro-Hom-pi-infty}, hence $\Hom_{\cO\llbracket K\rrbracket}^{\rm cont}(\mathrm{Im}(\theta^{\vee}\circ \varphi),\sigma^{\vee})\neq0$. In view of (\ref{equation-J-prime}), it contradicts that  $\varphi\in\mathrm{Ann}(M(\sigma))$. 

For the last assertion, we claim that $M(\sigma)$ is a cyclic $\widetilde{E}$-module. This implies that $M(\sigma)\cong \widetilde{E}/\mathrm{Ann}(M(\sigma))\cong \widetilde{E}/J$. However,  Lemma \ref{Lemma-M(sigma)} and the isomorphism (\ref{equation-J-prime}) give an exact sequence  
\begin{equation}\label{equation-S-regular}
0\ra M(\sigma) {\ra} M(\sigma)\ra \Hom_{\cO\llbracket K\rrbracket}^{\rm cont}(\pi_1^{\vee},\sigma^{\vee})^{\vee}\ra0.\end{equation}
Since $\Hom_{\cO\llbracket K\rrbracket}^{\rm cont}(\pi_1^{\vee},\sigma^{\vee})^{\vee}\cong\F$, we get $M(\sigma)\cong \F\llbracket S\rrbracket$ hence $\widetilde{E}/J\cong \F\llbracket S\rrbracket$. 

Now we prove the claim. Since we have a natural  isomorphism $M(\sigma)\otimes_{\widetilde{E}}\F\cong \Hom_{\cO\llbracket K\rrbracket}^{\rm cont}(\widetilde{P}\otimes_{\widetilde{E}}\F,\sigma^{\vee})^{\vee}$ by \cite[Proposition 2.4]{pa12},   it suffices to show that the latter space, whenever non-zero, is $1$-dimensional over $\F$ by Nakayama's lemma.    
By the projectivity of $\widetilde{P}$, we can find $x\in \widetilde{E}$ which makes the following diagram commutative:
\[\xymatrix{\widetilde{P}\ar@{->>}[d]\ar@{-->}^{x}[r]&\widetilde{P}\ar@{->>}[d]\\
\pi_{\infty}^{\vee}\ar^{S}[r]&\pi_{\infty}^{\vee}.}\] 
Applying $\Hom_{\cO\llbracket K\rrbracket}^{\rm cont}(*,\sigma^{\vee})^{\vee}$ to the diagram  and the cokernels
 we get using  (\ref{equation-sequence-S}) and (\ref{equation-J-prime}):
\begin{equation}\label{equation-dim=1}
\dim_{\F}\Hom_{\cO\llbracket K\rrbracket}^{\rm cont}(\widetilde{P}/x\widetilde{P},\sigma^{\vee})^{\vee}=\dim_{\F} \Hom_{\cO\llbracket K\rrbracket}^{\rm cont}(\pi_1^{\vee},\sigma^{\vee})^{\vee}=1.\end{equation}
Since $\widetilde{E}$ is a local ring by \cite[Corollary 2.5]{pa10} and $x$ is not an isomorphism (as $S$ is not surjective), $x$  lies in the maximal ideal of $\widetilde{E}$.  This implies a natural surjection $\widetilde{P}/x\widetilde{P}\twoheadrightarrow \widetilde{P}\otimes_{\widetilde{E}}\F$, and therefore \[\Hom_{\cO\llbracket K\rrbracket}^{\rm cont}(\widetilde{P}/x\widetilde{P},\sigma^{\vee})^{\vee}\twoheadrightarrow\Hom_{\cO\llbracket K\rrbracket}^{\rm cont}(\widetilde{P}\otimes_{\widetilde{E}}\F,\sigma^{\vee})^{\vee}.\]  This proves the claim using (\ref{equation-dim=1}).  
\end{proof}

\begin{corollary}\label{Coro-Ann(V)=J}
For $W$ a non-zero smooth $\F$-representation of $K$ of finite length,  $J$ is the only associated prime ideal of $M(W)$.
\end{corollary} 
\begin{proof}
It follows from Proposition \ref{Prop-J-prime} since $M(\cdot)$ is exact.
\end{proof}

\subsection{Colmez's functor}

We keep the notation of the preceding subsection. Recall that Colmez (\cite{co}) has  defined  an exact and covariant functor $\VV$ from the category of smooth, finite length representations of $G$ on $\cO$-torsion modules with a central character to the category of continuous  finite length representations of $G_{\Q_p}$ on $\cO$-torsion modules. Moreover, if $\pi$ is an object of finite length in $\mathrm{Mod}^{\rm sm}_{G,\zeta}(\cO)$, then the determinant of $\VV(\pi)$ is equal to $\epsilon\zeta$.  Following Pa\v{s}k\={u}nas \cite[\S3]{pa12}, we define an exact {covariant} functor $\check{\VV}:\mathfrak{C}_{\zeta}(\cO)\ra \mathrm{Rep}_{G_{\Q_p}}(\cO)$ as follows: for $M\in\mathfrak{C}_{\zeta}(\cO)$ of finite length, we let $\check{\VV}(M):=\VV(M^{\vee})^{\vee}(\epsilon\zeta)$ where $\vee{}$ denotes the Pontryagin dual. For general $M\in\mathfrak{C}_{\zeta}(\cO)$, write $M=\varprojlim M_i$ with $M_i$ of finite length in $\mathfrak{C}_{\zeta}(\cO)$ and define $\check{\VV}(M):=\varprojlim \check{\mathbf{V}}(M_i)$. 

\begin{proposition}\label{gl2theta}
The $G_{\Qp}$-representation $\check{\VV}(\pi_{\infty}^{\vee}(r,\lambda,\chi))$ is of rank $1$ over $\F\llbracket S\rrbracket $ and isomorphic to $\chi\mu_{S+\lambda}^{-1}$, where $\mu_{S+\lambda}:G_{\Q_p}\ra \F\llbracket S\rrbracket^{\times}$ is the unramified character sending geometric Frobenii to $S+\lambda$.
\end{proposition}
\begin{proof}
By the proof of \cite[1.5.9]{ki09}, $\VV(I_{\chi}(\Sym^r\F^2)/(T-\lambda)^{n})$ is isomorphic to the character
\[\chi\omega^{r+1}\mu_{S+\lambda}:G_{\Q_p}\ra (\F\llbracket S\rrbracket/S^n)^{\times}.\]
Using the fact that $\zeta$ reduces to $\chi^2\omega^r$, this implies by definition that 
\[\check{\VV}\bigl((I_{\chi}(\Sym^r\F^2)/(T-\lambda)^{n})^{\vee}\bigr)=(\chi\omega^{r+1}\mu_{S+\lambda})^{-1}\cdot(\omega\chi^2\omega^r)=\chi\mu_{S+\lambda}^{-1}.\]
The result follows by passing to the limit.
\end{proof}

As in \cite[\S1.5]{ki09}, denote by $\bar{\mathfrak{r}}$ the  pseudo-representation defined by
\[\chi\omega^{r+1}\mu_{\lambda}+\chi\mu_{\lambda^{-1}}\] 
and by $R^{\ps,\zeta}(\bar{\mathfrak{r}})$ the  universal pseudo-deformation ring with fixed determinant $\epsilon\zeta$ (see \S\ref{subsection-pseudo-def} for more details). It follows from results of \cite{pa10} that $\widetilde{E}\cong R^{\ps,\zeta}(\bar{\mathfrak{r}})$ if $(r,\lambda)\neq(p-2,\pm1)$ and $R^{\ps,\zeta}(\bar{\mathfrak{r}})\hookrightarrow \widetilde{E}$ otherwise (note that the definition of $R^{\ps,\zeta}(\bar{\mathfrak{r}})$ in \cite{pa10} is slightly different from ours).
Recall that Proposition \ref{Prop-J-prime} gives a surjective ring morphism $\widetilde{E}\twoheadrightarrow \F\llbracket S\rrbracket$, which we denote by $\tilde{\theta}$.
         
\begin{corollary}\label{kisintheta}
Assume (\textbf{H}) and moreover that $(r,\lambda)\neq (p-2,\pm1)$. Then, via the natural isomorphism $\widetilde{E}\cong R^{\ps,\zeta}(\bar{\mathfrak{r}})$, the map $\tilde{\theta}:\widetilde{E}\ra\F\llbracket S\rrbracket$ coincides with the map $\theta:R^{\ps,\zeta}(\bar{\mathfrak{r}})\ra \F\llbracket S\rrbracket$ constructed in \cite[1.5.11]{ki09}. 
\end{corollary}
\begin{proof}
The isomorphism $\widetilde{E}\cong R^{\ps,\zeta}(\bar{\mathfrak{r}})$ in \cite{pa10} is compatible with Colmez's functor, namely it is given by 
\[\check{\VV}: \widetilde{E}=\End_{\mathfrak{C}_{\zeta}(\cO)}(\widetilde{P})\cong\End_{G_{\Q_p}}(\check{\VV}(\widetilde{P}))\cong R^{\ps,\zeta}(\bar{\mathfrak{r}}).\]
The corollary follows since both $\tilde{\theta}$ and $\theta$ induce the same pseudo-deformation of $\overline{\mathfrak{r}}$ over $\F\llbracket S\rrbracket$ by Proposition \ref{gl2theta} and \cite[1.5.11]{ki09} (taking into account of the determinant).
\end{proof}

\section{The versal and pseudo-deformation rings}\label{maps}

Let $\overline{\rho}: G_{\Qp}\ra \GL_2(\FF)$ be a (continuous) representation. We aim to describe the versal deformation rings for various $\overline{\rho}$ explicitly, and then construct maps between them.  For these we follow the methods of \cite{bo} and \cite[Appendix B]{pa10}.

\subsection{The versal deformation rings}

We refer the reader to \cite{maz1} for the general theory of Galois deformations. The deformation functor $D(\overline{\rho})$ on the category of Artinian local $\cO$-algebras with residue field  $\FF$    always has a versal hull $R^{\rm ver}=R^{\rm ver}(\overline{\rho})$, which is unique up to \emph{non-canonical} isomorphisms.  

In the rest of this section, we always assume that $\rhobar$ is of the form 
$\smatr{\ide}{*}0{\omega}$ or $\smatr{\ide}{0}{*}{\omega}$. It is obvious 
that the deformation functor $D(\overline{\rho})$ is representable by a universal deformation ring if $\overline{\rho}$ is non-split, and has only a versal hull otherwise.

Denote by $L\subset \overline{\QQ}_p$ the fixed field of $\mathrm{Ker}(\overline{\rho})$, and write $H=\Gal(L/\Qp)$.  Let $U\subset H$ be its $p$-Sylow subgroup, which is isomorphic to $\F_p$ if $\overline{\rho}$ is a non-split extension, and is trivial otherwise. Write $F$ as the fixed field of $\Ker \omega$. Then the quotient $C=\Gal(F/\Qp)$ is isomorphic to $\F_p^{\times}$.  For a deformation $\rho_A$ to $(A,\fm)$, the image of $G_F\subset G_{\Qp}$ thus has diagonal entries lying in $1+\fm$ and the lower left (resp. upper right)  entry lying in $\fm$, hence $\rho_A$ factors through $\Gal(F(p)/\Qp)$, with $F(p)$ the composition of all the finite extensions of $F$ whose degrees are powers of $p$.  As the order of $C$ is prime to $p$, we can and do fix an isomorphism
\[\Gal(F(p)/\Qp)\cong G_F(p)\rtimes C.\] We regard $C$ as a subgroup of $\GL_2(R)$ for any complete noetherian  local ring $R$, via the map \[g\mapsto \matr{1}00{\tilde{\omega}_C(g) }
 \]where $\tilde{\omega}_C: C\ra \Zp^{\times}$ is the Teichm\"{u}ller lifting of $\omega|_C: C\ra \F_p^{\times}$.

A pro-$p$ group $D$ is called a Demu\v{s}kin group if $\mathrm{dim}_{\Fp}H^1(D, \Fp):=n(D)<\infty$, $\mathrm{dim}_{\Fp}H^2(D, \Fp)=1$, and the cup product \[H^1(D, \Fp)\times H^1(D, \Fp)\stackrel{\cup}{\ra} H^2(D, \Fp)\]is a non-degenerate bilinear form.
Since we assume $p\neq 2$, the Demu\v{s}kin group $D$ is determined (up to isomorphism) by $n(D)$ and $t(D)$, where $t(D)$ denotes the order of the torsion subgroup of $D^{\rm ab}$. Namely, $D$ is isomorphic to the pro-$p$ group with $n(D)$ generators $t_1,\cdots, t_{n(D)}$ and one relation \[t_1^{t(D)}[t_1,t_2][t_3,t_4]\cdots[t_{n(D)-1},t_{n(D)}]=1,\] where $[t_i,t_j]=t_i^{-1}t_j^{-1}t_it_j$ are commutators; see \cite[Theorem 3]{la}. It is well-known (see e.g. \cite[Theorem 7]{la}) that $G_F(p)$ is  a Demu\v{s}kin group, for which  $n(G_F(p))=p+1$ and $t(G_F(p))=p$.

For a pro-$p$ group $\cF$,  define a filtration $\{\cF_i\}_{i\geq 1}$ by setting \[\cF_1=\cF, \qquad \cF_i=\cF_{i-1}^p[\cF_{i-1},\cF],\qquad \mathrm{gr}_i\cF=\cF_i/\cF_{i+1}.\] The Frattini quotient $\mathrm{gr}_1\cF$ will play an important role in the following. By \cite[Lemma 3.1]{bo} the action on $\cF$ of a group of order prime to $p$ is determined by the action on $\mathrm{gr}_1\cF$, up to inner automorphisms of $\cF$.

We choose $\cF$ to be a free pro-$p$ group in $p+1$ generators, and a surjection \[\phi: \cF\twoheadrightarrow G_F(p)\] whose kernel $\cR$ is generated by a single element $r\in \cF$. We see that $\mathrm{gr}_1\cF\simeq \mathrm{gr}_1G_F(p)$, hence $r\in \cF_2$. By \cite[Lemma 3.1]{bo}, the $C$-action on $G_F(p)$ extends uniquely to $\cF$ and makes $\phi$ a $C$-equivariant homomorphism, hence gives a homomorphism 
\begin{equation}\label{phic}\phi: \cF\rtimes C\ra G_F(p)\rtimes C\simeq\Gal(F(p)/\Qp).\end{equation}  We will relate $r$ with  the Demu\v{s}kin relation.

The local class field theory  and the $C$-module structure of  $G_F(p)^{\rm ab}$ determined by Iwasawa \cite[Theorem 1]{iw} give 
the following result.

\begin{lemma} \label{pplus1}
There is a natural isomorphism of $\Fp[C]$-modules\[\mathrm{gr}_1\cF\simeq \mathrm{gr}_1G_F(p)\simeq \Fp\oplus\mu_p\oplus \Fp[C]\] such that $\mu_p$ is the image of the torsion subgroup of $G_F(p)^{\rm ab}$ under the projection $G_F(p)^{\rm ab}\twoheadrightarrow \mathrm{gr}_1G_F(p)$ on which $C$ acts by $\omega$, and $\Fp[C]$ is the image of the $2$nd ramification subgroup $I_{F,2}$ of the inertia $I_F$. 
\end{lemma}

({\bf GEN}) Fix generators $\xi_0,\cdots,\xi_p$ of $\mathrm{gr}_1G_F(p)$ so that  $\xi_1$ generates $\mu_p$ and $\xi_2,\cdots,\xi_p$ generate $\Fp[C]$, and such that  $C$ acts on $\xi_i$ by $\omega^{i}$. \medskip

We remark that Lemma \ref{lapa} below is the best one can achieve, when choosing generators of $\cF$ that respect both $C$-actions and the Demu\v{s}kin relation; cf. \cite[Proposition 3.6]{bo}.

\begin{lemma}\label{lapa}

There exist generators $t_0,\cdots,t_p$ in $\cF$ lifting  $\xi_0,\cdots,\xi_p$ such that

(i) $\forall i\in \{0,\cdots,p\}, \forall g\in C$, we have $gt_ig^{-1}=t_i^{\tilde{\omega}_C^i(g)}$.

(ii) The element $r_D:=t_1^p[t_0,t_p][t_1,t_{p-1}]\cdots[t_{\frac{p-1}{2}},t_{\frac{p+1}{2}}]$ is congruent to $r$ modulo $\cF_3$.

\end{lemma}

\begin{proof}
Take a lifting $t_0,\cdots,t_p\in \cF$ so that the $C$-actions are as in (i), which is achievable as $C$ is of order prime to $p$; recall \cite[Lemma 3.1]{bo}. That they may be chosen to satisfy (ii) follows essentially from \cite[Proposition 3]{la} (see also \cite[Lemma B.1]{pa12}), where it is shown how  the cup product $H^1(D,\Fp)\times H^1(D,\Fp)\ra H^2(D,\Fp)$ for a Demu\v{s}kin group $D$ is determined by the image of an element of $\cF_2$ modulo $\cF_3$. Namely, the image in  $\cF_2/\cF_3$ of such an element must be of the form of Demu\v{s}kin relation (up to rescaling), and it defines an isomorphism $H^2(D,\Fp)\simto \Fp$. It is then easy to see that $r_D$  defines the same cup product on $H^1(G_F(p),\Fp)\times H^1(G_F(p),\Fp)$ as that defined by $r$, hence has the same image as $r$ modulo $\cF_3$.\end{proof} 


To construct the (uni-)versal deformations for some $\rhobar$ with semi-simplification $\mathbbm{1}\oplus\omega$,  we first introduce the following general result.

\begin{proposition}\label{R}

Let $R$ be a complete noetherian local $\cO$-algebra with residue field $\FF$. Suppose there are matrices $m_i$ in $\GL_2(R)$  which satisfy the following conditions:

(1) $C$-actions: $gm_ig^{-1}=m_i^{\tilde{\omega}_C^i(g)}, \forall g\in C$.
 
(2) Demu\v{s}kin relation: $m_1^p[m_0,m_p][m_1,m_{p-1}]\cdots[m_{\frac{p-1}{2}},m_{\frac{p+1}{2}}]=1$.

\noindent Then we have

(i) The assignment $t_i\mapsto m_i$  ($i=0,\cdots, p$) is a $C$-equivariant group homomorphism, hence defines a homomorphism $\alpha_R: \cF\rtimes C \ra \GL_2(R)$, which satisfies  that $\alpha_R(r)\in \alpha_R(\cF_3)$.

(ii) There is a continuous homomorphism \[\rho_R: \Gal(F(p)/\Qp)\ra \GL_2(R)\] and a continuous homomorphism \[\phi': \cF\rtimes C\ra \Gal(F(p)/\Qp)\] with the properties $\Ker \phi'\in \cF_2$, $\alpha_R=\rho_R\circ \phi'$, and $\phi'\equiv \phi$ mod $\cF_3$.  

(iii) Moreover,  $\phi'$ can be chosen uniformly for  various $R$ if $\cap_{R}\Ker{\alpha_R}$ is non-empty.

\end{proposition}

\begin{proof}

(i) follows from (1) and (2); we have $\alpha_R(r)\in \alpha_R(\cF_3)$  by Lemma \ref{lapa}(ii)  and that $r_D\in \Ker\alpha_R$. 

(ii) and (iii) can be obtained by  the proof of \cite[Proposition 3.8]{bo}, with $\Ker\alpha$ \emph{loc.cit.} replaced by the intersection of the $\Ker\alpha_R$'s; the intersection is taken for the uniformness of $\phi'$. More precisely, we first note that  $C$ acts on $H^2(G_F(p),\Fp)$ by $\omega^{-1}$, since the latter is the $\Fp$-dual of  $\mu_p$ on which $C$ acts by $\omega$. Thus, by the discussion on \cite[Page 118]{bo}, $C$ acts on $r$ by $\tilde{\omega}$. Now, for any $i\geq 2$, form the composite $N_i$ of $\cF_i$ and $\cap_{R}\Ker{\alpha_R}$. Then \cite[Lemma 3.2]{bo} shows that there is an element $r_i\in N_i$ on which $C$ acts as $\tilde{\omega}$, and $r_i\equiv r$ mod $\cF_i$ for any $i\geq  2$, hence all $r_i\in \cF_2$. 

Denote by $C_{r_i}$ the closure of $\{r_j\}_{j\geq i}$ in $N_i\cap \cF_2$. Then $I:=\cap_{i\geq 2}C_{r_i}$ is non-empty by the compactness of $\cF$, and lies in $(\cap_{R}\Ker{\alpha_R})\cap \cF_2$. 
Note that $C$ acts on any element in $C_{r_i}$ (for any $i\geq 2$) via $\tilde{\omega}$, because the set $\{x\in \cF|g\cdot x=x^{\tilde{\omega}(g)}, \forall g\in G\}$ is closed. Thus $C$ acts on any element in $I$ via $\tilde{\omega}$. Furthermore, an element in $I$ is congruent to $r$ modulo $\cF_3$ by the construction of $r_i's$, hence $\cF$ modulo such an element defines a Demu\v{s}kin group which is isomorphic to $G_F(p)$, by \cite[Proposition 3]{la}. Then, $\cF$ modulo an element in $I\subset \cF_2$ gives the wanted homomorphisms $\varphi'$ and $\rho_R$. 
\end{proof}

\medskip

Depending on the shapes of the representations $\overline{\rho}$, the (candidates for) versal deformations and deformation rings are listed  below.   

 For each $R=R^{\rm ver}, R^1, R^{\rm peu}$  and the matrices $m_i$ in $\GL_2(R)$ below, direct computation shows that the conditions (1) and (2) in Proposition \ref{R} are satisfied. (We refer the reader to \cite[Lemma 5.3 (i)-(iii)]{bo} for more details on the choices of these matrices.) Moreover, the intersection $\cap_{R}\Ker{\alpha_R}$ of these rings is non-empty, because, for instance, $t_2$ lies in it. Therefore Proposition \ref{R} applies.

\subsubsection{The split  case}\label{splitng}

Let $\overline{\rho}$ be  $\mathbbm{1}\oplus\omega$. We pick indeterminate variables $a_0, a_1, b, c_0, c_1$, $d_0, d_1$ and write

\[m_0=(1+a_0)^{1/2}
\left(
\begin{array}{cc}
  (1+d_0)^{1/2} &  0 \\
 0 &     (1+d_0)^{-1/2}
\end{array}
\right), \qquad m_1=
\left(
\begin{array}{cc}
  1 &  0 \\
 c_0 &   1
\end{array}
\right), \]\[
m_{p-1}=(1+a_1)^{1/2}
\left(
\begin{array}{cc}
  (1-p+d_1)^{1/2}&  0 \\
 0 &     (1-p+d_1)^{-1/2}
\end{array}
\right), \quad
m_p=
\left(
\begin{array}{cc}
 1 &  0 \\
 c_1 &   1
\end{array}
\right),
\]\[m_{p-2}=
\left(
\begin{array}{cc}
  1&   b   \\
  0& 1 
\end{array}
\right),\quad m_2=\cdots=m_{p-3}=I_{2\times 2}.\]
Set \[R^{\rm ver}=\frac{\cO\llbracket a_0,a_1,b, c_0,c_1,d_0,d_1\rrbracket}{(c_0d_1-c_1d_0)}.\]
By the description above, the same proof as in \cite[B.4]{pa10}  shows that the reducibility ideal of $R^{\rm ver}$ is $(bc_0,bc_1)$, that is, for $x:R^{\ver}\ra E$ a closed point, the corresponding deformation $\rho_x$ is reducible if and only if $(bc_0,bc_1)\subset \Ker x$.

 \subsubsection{The non-split cases}\label{nsplitng}

(1) Assume $\overline{\rho}$ is a  non-split extension of $\omega$ by $\mathbbm{1}$ (unique up to scalar as $p\geq 5$). Pick indeterminate variables $a_0,a_1, c_0,c_1, d_0,d_1$. Set $m_0,\cdots,m_p$ as in the split case, except that we replace $m_{p-2}=\smatr{1}{b}01$ with $m_{p-2}=\smatr{1}101$.

Set \[R^1=\frac{\cO\llbracket a_0,a_1,c_0,c_1,d_0,d_1\rrbracket}{(c_0d_1-c_1d_0)}.\]
One sees easily that the reduction of $\rho_{R^1}$ given by Proposition \ref{R} is a non-split extension of $\omega$ by $\mathbbm{1}$. Similarly as before, we have that the reducibility ideal of $R^1$ is $(c_0,c_1)$.\medskip

(2)  Assume $\overline{\rho}$ is a non-split extension $0\ra \omega\ra\overline{\rho}\ra\ide\ra0$. We know that $\mathrm{Ext}_{G_{\Qp}}^1(\mathbbm{1},\omega)\simeq  H^1(G_{\Qp},\omega)$ is of dimension $2$, so $\overline{\rho}$ could be either peu ramifi\'{e} or tr\`{e}s ramifi\'{e} extensions,  as defined by Serre  \cite[\S2.4]{se}. We recall the definition below.

Write $K_0=\Qp^{\rm ur}$ the maximal unramified extension of $\Qp$, and $K_t=K_0(\mu_p)$ the tamely ramified field. Then Kummer theory tells us that the Galois representation $\rhobar|_{\Gal(\bQp/K_t)}$ must factor through $\Gal(K/K_t)$ for some $K$ of the form 
\[K=K_t(x_1^{1/p},\cdots,x_m^{1/p})\quad \text{for}\quad x_i\in K_0^{\times}/(K_0^{\times})^p,\]for some $m\geq 1$.
We then say $\rhobar$ is \emph{peu ramifi\'{e}} if $p|v_p(x_i)$ for each $i$, and say the associated element in  $H^1(G_{\Qp},\omega)$ is a peu ramifi\'{e} extension. A peu ramifi\'{e} extension is unique up to scalars. Depending on context, we sometimes call the trivial extension $\mathbbm{1}\oplus\omega$ a peu ramifi\'{e} extension. All the other extensions are called \emph{tr\`{e}s ramifi\'{e}} extensions.

The following equivalent variation  of Serre's definition is easy to obtain.

\begin{lemma}\label{peuline}
An extension $0\ra \omega\ra \overline{\rho}\ra \mathbbm{1}\ra 0$ is peu ramifi\'{e} if and only if the image of $2$nd ramification subgroup $I_{F,2}\subset \Gal(\bQp/K_t)$ under $\rhobar$ is trivial.
\end{lemma}
\begin{proof}

First  recall from Lemma \ref{pplus1}  that the image of $I_F$ (or equivalently, that of the wild inertia $I_{F,1}$) in $\mathrm{gr}_1G_F(p)$ is isomorphic to $\mu_p\oplus \Fp[C]$ as $\Fp[C]$-modules, and that  under the same reduction the $2$nd ramification subgroup $I_{F,2}$ is mapped onto $\Fp[C]$ and  the $p$-torsion subgroup of $I_{F,1}$ is mapped onto $\mu_p$.

Let $H$ be the kernel of the projection $I_{F,1}\twoheadrightarrow \mu_p\subset \mathrm{gr}_1G_F(p)$ and $K$ be the fixed field of $H$. Then $H=I_{F,2}$ and $K$ is an abelian extension over $K_t$ of degree $p$. Moreover, since $K_t$ contains the $p$-th roots of unity, $K$ is a Kummer extension and of the form $K=K_t(u^{1/p})$ with $u\in K_0^{\times}/(K_0^{\times})^p$. We then have the $2$nd ramification subgroup $\Gal(K/K_t)_2=\{1\}$ by \cite[p.68, Corollary]{se2}. On the other hand, it is elementary to check that $\Gal(K_t(u^{1/p})/K_t)_2=\{1\}$ if and only if  $v_p(u)=0\ (\mathrm{mod}\ p)$. The claim follows.\end{proof}

\begin{remark}By Kummer theory, we have the isomorphism\[\Qp^{\times}/(\Qp^{\times})^p\simto H^1(G_{\Qp},\omega), \quad u\mapsto (g\mapsto g(u^{1/p})/u^{1/p}).\] Then the image of $\Zp^{\times}/(\Zp^{\times})^p$ in $H^1(G_{\Qp},\omega)$ is the peu ramifi\'{e} line. Hence a peu ramifi\'{e} extension $\rhobar$ must factor through $\Gal(\Qp(\mu_p,(1-p)^{1/p})/\Qp)$; take $u=1-p$ in the proof of Lemma \ref{peuline}. \end{remark}

Assume $\overline{\rho}$ is  a  non-split extension of $\mathbbm{1}$ by $\omega$ which is peu ramifi\'{e}. 
We pick indeterminate variables $a_0,a_1,x_1,x_2,x_3$ and write

\[m_0=(1+a_0)^{1/2}
\left(
\begin{array}{cc}
  (1+x_1)^{1/2} &  0 \\
 0 &     (1+x_1)^{-1/2}
\end{array}
\right), \qquad m_1=
\left(
\begin{array}{cc}
  1 &  0 \\
 1 &   1
\end{array}
\right), \]\[
m_{p-1}=(1+a_1)^{1/2}
\left(
\begin{array}{cc}
 (1-p+x_1x_2)^{1/2} &  0 \\
 0 &     (1-p+x_1x_2)^{-1/2}
\end{array}
\right), \quad
m_p=
\left(
\begin{array}{cc}
 1 &  0 \\
 x_2 &   1
\end{array}
\right),
\]\[m_{p-2}=
\left(
\begin{array}{cc}
  1&   x_3   \\
  0& 1 
\end{array}
\right),\quad
m_2=\cdots=m_{p-3}=I_{2\times 2}.\]Set \[R^{\rm peu}=\cO\llbracket a_0,a_1,x_1,x_2,x_3\rrbracket.\]The reducibility ideal of $R^{\rm peu}$ is $(x_3)$.

By Lemma \ref{peuline}, ({\bf GEN}) and the choices $m_i$, the deformation $\rho_{R^{\rm peu}}$ obtained via Proposition \ref{R} reduces to the peu ramifi\'{e} extension $\overline{\rho}$ modulo the maximal ideal of $R^{\rm peu}$ (up to isomorphism).  This justifies the notation $\rho_{R^{\rm peu}}$.

\medskip
 
\begin{corollary}\label{Risversal}

The rings $R=R^{\rm ver}, R^1, R^{\rm peu}$ in \S\S\ref{splitng}, \ref{nsplitng} are the (uni-)versal deformation rings of the corresponding $\overline{\rho}$, and  the continuous homomorphisms $\rho_R$ obtained via Proposition \ref{R}   are the associated (uni-)versal deformations.

\end{corollary}

\begin{proof}
This is by the same proof as in  \cite[Theorem 6.2]{bo}.
\end{proof}

We need to consider the deformations with fixed determinants, which is needed to link the deformation rings to $p$-adic Langlands correspondence.

\begin{corollary}\label{det} Keep the notation above.
Let $\psi: G_{\Qp}\ra \cO^{\times}$ be a continuous character whose reduction mod $\varpi$ is equal to $\mathbbm{1}$, and let $D(\overline{\rho})^{\psi}$ be the sub-functor of $D(\overline{\rho})$ parametrizing the deformations with determinants equal to $\epsilon\psi$. Then the functor $D(\overline{\rho})^{\psi}$  is (pro-)represented by the quotient of $R$ by $(a_0-\alpha_0,a_1-\alpha_1)$ for some $\alpha_0,\alpha_1\in \varpi\cO$.   
\end{corollary}

\begin{proof}

This is clear by the choice of the matrices $m_i=\alpha_R(t_i)$. 
\end{proof}

\subsection{Comparison of various deformation rings}
\label{subsection-pseudo-def}

 Let $\overline{\rho}: G_{\Qp}\ra \GL_2(\FF)$ be the representation as before. 
 Define $D^{\rm ps}=D^{\rm ps}(\mathrm{tr}\overline{\rho})$ as the functor from the category of Artinian local $\cO$-algebras $A$ with residue field $\FF$ to the category of sets of pseudo-deformations of $\mathrm{tr}\overline{\rho}$, which is always  (pro-)represented by a complete noetherian local $\cO$-algebra $R^{\rm ps}=R^{\rm ps}({\mathrm{tr}\overline{\rho}})$,   equipped with the universal pseudo-deformation $T^{\rm univ}$.  Furthermore, we define $D^{\ps, \psi}$ to be the sub-functor of $D^{\rm ps}$ parametrizing the pseudo-deformations $T\in D^{\rm ps}(A)$ such that $\epsilon\psi(g)$ is mapped to 
 $\frac{T^2(g)-T(g^2)}{2}$ under the structure morphism. The noetherian local $\cO$-algebra representing $D^{\ps, \psi}$ is denoted by $R^{\ps, \psi}$ and the corresponding universal pseudo-deformation is denoted by $T^{\rm univ,\psi}$.

\medskip
  
By the constructions in \S\S \ref{splitng}, \ref{nsplitng} and Proposition \ref{R}, we will write down the maps among various (pseudo-)deformation rings, adapting the idea of \cite[Appendix B]{pa10}. \medskip  

 \subsubsection{The map $f^1$}
  \label{subsubsection-f1}
 
First consider a non-split  extension $0\ra \mathbbm{1}\ra \overline{\rho}^1\ra \omega\ra 0$. 
The construction of $R^1=R^{\ver}({\overline{\rho}^1})$ provides  the following description of the pseudo-deformation ring $R^{\rm ps}=R^{\rm ps}({\tr\overline{\rho}})$.

 \begin{proposition}\label{ngenericps}
 
 The natural homomorphism $R^{\ps}\ra R^1$ given by taking traces is an isomorphism and induces the isomorphism  \[f^1: R^{\rm ps,\psi}\simeq R^{1,\psi}.\]

 \end{proposition}

 \begin{proof}
 This is \cite[Proposition B.15]{pa10}.
 \end{proof}

We identify  $R^{\rm ps,\psi}=R^{1,\psi}$ from now on, hence have an isomorphism \begin{equation}\label{equation-Rps-isom}
R^{\ps,\psi}\cong \cO\llbracket c_0,c_1,d_0,d_1\rrbracket/(c_0d_1-c_1d_0).\end{equation}

Recall that we have defined in \S\ref{subsection-J} a prime ideal $J$ of $R^{\ps,\psi}$, the kernel of the map $\theta: R^{\ps,\psi}\twoheadrightarrow \F\llbracket S\rrbracket$. (Here we have  taken $\zeta$ \emph{loc. cit.} to be $\psi$, whose reduction mod $\varpi$ is trivial.)

 \begin{lemma}\label{frob}
Under the identification (\ref{equation-Rps-isom}), we have $J=(\varpi, c_0,c_1,d_1)$.
  \end{lemma}

 \begin{proof}
First, $\varpi\in J$ as  the image of $\theta$ is $\FF\llbracket S\rrbracket$. Since $c_0,c_1$ lie in the reducibility ideal, they lie in $J$. By Proposition \ref{gl2theta},  the image of the inertia $I_{F}$ under $\theta\circ T^{\rm univ,\psi}$ is trivial as $\chi$ is trivial in our case. By Lemma \ref{pplus1} and the choice ({\bf GEN}) of generators of $\mathrm{gr}_1\cF\simeq \mathrm{gr}_1G_{F}(p)$ (and \cite[Lemma 3.1]{bo} again), the image $t_{p-1}'=\phi'(t_{p-1})\in G_F(p)$ of $t_{p-1}\in \cF$  comes from $I_{F}$, hence has trivial action under $\theta\circ T^{\rm univ,\psi}$.  Thus we have $\theta(d_1)=0$, noting that $T^{\rm univ,\psi}(t_{p-1}')= (1+\alpha_1)^{1/2}((1-p+d_1)^{1/2}+(1-p+d_1)^{-1/2})$ with $\alpha_1\in \varpi\cO$ by Corollary \ref{det}, and $\theta(p)=0$. We thus get the inclusion $(\varpi,c_0,c_1,d_1)\subseteq J$, from which  the result follows since they have the same height.
\end{proof}

\subsubsection{The map $f^{\rm peu}$}
\label{subsubsection-fpeu} 
 
Let $\overline{\rho}^{\rm peu}$ be a (non-split) peu ramifi\'{e} extension.  By the construction of $R^{\rm peu,\psi}\cong\cO\llbracket x_1,x_2,x_3\rrbracket$  in \S\ref{nsplitng}, the ideal of $R^{\rm peu,\psi}$ generated by the $(1,2)$-entry of $\rho_{R^{\rm peu,\psi}}(g)$ for all $g\in G_{\Q_p}$ is just $(x_3)$, so the conjugation 
\[\smatr{x_3^{-1}}001\rho_{R^{\rm peu,\psi}}\smatr{x_3}001\]
  still  takes values in $R^{\rm peu,\psi}$. We check easily that this gives a representation on $R^{\rm peu,\psi}$ whose residual representation is a non-split extension of $\omega$ by $\mathbbm{1}$, hence induces a ring homomorphism $R^{1,\psi}\ra R^{\rm peu,\psi}$. It is seen at the same time that $\smatr{x_3^{-1}}001\rho_{R^{\rm peu,\psi}}\smatr{x_3}001$ is isomorphic to the base change to $R^{\rm peu,\psi}$ of the universal representation on $R^{1,\psi}$. By Proposition \ref{R}(iii) and the fact that taking conjugation does not change traces, the composition of the above map with (\ref{equation-Rps-isom}) gives us the trace map:
 \begin{equation}\label{pspeu} 
f^{\rm peu}:R^{\rm ps,\psi} \simeq\frac{\cO\llbracket c_0,c_1,d_0,d_1\rrbracket}{(c_0d_1-c_1d_0)}\hookrightarrow  R^{\rm peu,\psi},\end{equation}
 \[c_0\mapsto x_3, \quad c_1\mapsto x_2x_3, \quad d_0\mapsto x_1,\quad d_1\mapsto x_1x_2.\]
  
\subsubsection{The map $f^{\ver}$} 
\label{subsubsection-fver}
 
Assume $\overline{\rho}=\mathbbm{1}\oplus \omega$ is split. As in \S\ref{subsubsection-fpeu}, one checks, using the construction in \S\S\ref{splitng}, \ref{nsplitng} and Proposition \ref{R}, that the conjugation by $\smatr{b}001$  on the universal representation $\rho_{R^{\rm ver,\psi}}$ gives a map $R^{1,\psi}\ra R^{\rm ver,\psi}$, hence the trace map:
 \begin{equation}\label{ngf}f^{\rm ver}: R^{\rm ps,\psi} \ra R^{\rm ver,\psi}, \quad c_i\mapsto bc_i, \quad d_i\mapsto d_i, \quad i=0,1. \end{equation}

By Lemma \ref{frob}, $R^{\rm ver,\psi}/JR^{\rm ver,\psi}$ has three minimal prime ideals:
\begin{equation}\label{threeideals}\fp_1=(\varpi, c_0,c_1,d_1),\quad \fp_2=(\varpi,b,c_1,d_1),\quad \fp_3=(\varpi,b,d_0,d_1).\end{equation}
In fact, one checks that $JR^{\ver,\psi}=\fp_1\cap \fp_2\cap\fp_3$. 
Write 
 \begin{equation}\label{ngf1}f_i^{\rm ver}: R^{\rm ps,\psi}\ra   R_{\fp_i}^{\rm ver,\psi}\end{equation} 
for the induced homomorphism. The following property of $f_1^{\ver}$ and $f_2^{\ver}$ will be used in the proof of the  Breuil-M\'ezard conjecture later.

\begin{proposition}\label{f1f2}
For $i=1,2$, $f_i^{\ver}$  is flat, 
and for any radical ideal $\fa$ of $R^{\ps}$, $\fa R^{\rm ver}_{\fp_i}$ is still radical.
\end{proposition}
\begin{proof}
 
We only prove the claim for $\fp_1$; the proof for $\fp_2$ goes over verbatim. Note that  $R^{\ps,\psi}_J$ is a regular local ring, because its Krull dimension is $3$ and  its maximal ideal is generated by $\varpi, c_0,d_1$ (as $c_1=c_0d_1d_0^{-1}$). Also, $R^{\ver,\psi}_{\fp_1}$ is Cohen-Macaulay since it is a localization of a Cohen-Macaulay ring.  
Since $(f^{\ver})^{-1}(\fp_1)=J$, the map  
$f_1^{\rm ver}: R^{\rm ps,\psi}\ra   R_{\fp_1}^{\rm ver,\psi}$
factors as 
\[R^{\ps,\psi}\hookrightarrow R^{\ps,\psi}_J\ra R^{\ver,\psi}_{\fp_1},\]
where the second map is  a local homomorphism. The first map is clearly flat.  The second map is also flat by \cite[Theorem 23.1]{mat}, since one checks directly that $\dim R^{\ver,\psi}_{\fp_1}=\dim R^{\ps,\psi}_J+\dim  R^{\ver,\psi}_{\fp_1}/J R^{\ver,\psi}_{\fp_1}=3$. In fact, since $JR^{\ver,\psi}_{\fp_1}=\fp_1R^{\ver,\psi}_{\fp_1}$, the quotient ring $R^{\ver,\psi}_{\fp_1}/JR^{\ver,\psi}_{\fp_1}$ is a field. Thus the map $f_1^{\rm ver}$ is flat.

 Recall \cite[Theorem 2.1]{Io}: Let $u: A\ra B$ be a local flat morphism of noetherian local rings, with $A$ a Nagata ring. If $B/\fm_AB$  is a geometrically reduced $A/\fm_A$-algebra, then $u$ is a reduced morphism (see \cite[Definition 1.1]{Io} for its definition), which in particular sends radical ideals to radical ideals.

For the second assertion, we first note that  $\fa R^{\ps,\psi}_J$ is a radical ideal.  Hence, it suffices to check that the map $R^{\ps,\psi}_J\ra R^{\ver,\psi}_{\fp_1}$ sends radical ideals to radical ideals. The last map is a flat local morphism of noetherian local rings. The ring  $R^{\ps,\psi}_J$ is a Nagata ring, since it is a localization of a complete noetherian local ring; see \cite[Chapitre IX, \S4, n$^{\circ}$4]{bour}. 
By  \cite[Theorem 2.1]{Io},  we only need to show that $R^{\ver,\psi}_{\fp_1}/J R^{\ver,\psi}_{\fp_1}\simeq  \FF(\!( d_0,b)\!)$, the  field of fractions of $R^{\ver,\psi}/\fp_1\simeq \FF\llbracket d_0, b\rrbracket$, is geometrically reduced over $R^{\ps,\psi}_J/JR^{\ps,\psi}_J\simeq\FF(\!(d_0)\!)=:k$. To see this, let $k'$ be any finite extension of $k$. Then $k'\otimes_kk(\!(b)\!)$ is reduced since it is a field. But we have the inclusion $k'\otimes_k\FF(\!( d_0,b)\!)\subset k'\otimes_kk(\!(b)\!)$ by the flatness of $k'$ over $k$, which implies that $k'\otimes_k\FF(\!( d_0,b)\!)$ is also reduced. 
\end{proof}

\begin{remark}
One sees easily that the induced homomorphism $f_3^{\ver}: R^{\ps,\psi}\ra R^{\ver,\psi}_{\fp_3}$ is not flat.
\end{remark}
 
 \begin{remark}\label{flatgeneric}

 In the case that $\brho$ is split generic, that is, $\brho\cong\chi_1\oplus\chi_2$ with $\chi_1\chi_2^{-1}\notin\{\ide,\omega^{\pm1}\}$, the situation is similar and in fact simpler. More precisely, using the machinery above, one gets, after choosing parameters, that $R^{\rm ps,\psi}=\cO\llbracket y_1, y_2,y_3\rrbracket$ and $R^{\rm ver,\psi}=\cO\llbracket b,y_1, y_2,y_3\rrbracket$.  By a similar construction as in \S\S \ref{splitng}, \ref{nsplitng}, taking traces induces the homomorphism 
 \[
 f^{\rm ver}: R^{\rm ps,\psi}\hookrightarrow R^{\rm ver,\psi},\quad y_1\mapsto y_1, \quad y_2\mapsto y_2,  \quad y_3\mapsto by_3.\]
 One then sees that $f^{\rm ver}$ is flat and maps radical ideals to radical ideals.

\end{remark}

\subsubsection{The maps $\gamma_i$}
\label{subsubsection-gamma}

Consider the ideals $\fp_2$ and $\fp_3$ of $R^{\ver,\psi}$. Meanwhile, one checks that $R^{\rm peu}/JR^{\rm peu}$ has two minimal prime ideals, which we denote by $\fq_2,\fq_3$ (notation chosen to be consistent with $\fp_2,\fp_3$):
\[\fq_2=(\varpi,x_2,x_3),\quad \fq_3=(\varpi,x_1,x_3).\]


\begin{proposition}\label{f3}  Let    $\widehat{R_{\fp_i}^{\rm ver,\psi}}$ ($i=2,3$) be the completion of $R_{\fp_i}^{\rm ver,\psi}$ with respect to its maximal ideal. We still write $f^{\ver}$ for the composition $R^{\ps,\psi}\ra R^{\ver,\psi}\ra \widehat{R^{\ver,\psi}_{\fp_i}}$.  

(i) There is a unique local homomorphism of $\cO$-algebras \begin{equation}\label{peus}
\gamma_i: R^{\rm peu,\psi}_{\fq_i}\ra  \widehat{R_{\fp_i}^{\rm ver,\psi}}\end{equation}
which is compatible with the trace maps $f^{\rm peu}$  and $f^{\rm ver}$. That is, we have the following commutative diagram: 
 \begin{equation}\label{equation-diagram}
\xymatrix{R^{\rm ps,\psi}\ar_{f^{\rm peu}}[d]\ar^{f^{\ver}}[dr] \\
 R^{\rm peu,\psi}_{\fq_i}\ar[r]^{\gamma_i}&\widehat{R_{\fp_i}^{\rm ver,\psi}}.}\end{equation}

(ii) The map $\gamma_i$ is flat and  sends radical ideals to radical ideals.

\end{proposition}

\begin{proof}

 (i)  Define  \[\gamma_i: R^{\rm peu,\psi}\ra  \widehat{R_{\fp_i}^{\rm ver,\psi}}, \quad x_1\mapsto d_0, \quad x_2\mapsto c_0^{-1}c_1, \quad x_3\mapsto bc_0.\]
One checks that this is  well-defined. Now look at the inverse image of the maximal ideal $\fp_i\widehat{R_{\fp_i}^{\rm ver,\psi}}$ in $R^{\rm peu,\psi}$, which is a prime ideal   containing $\fq_i$ but not the maximal ideal (because it does not contain $x_1$ (resp. $x_2$) when $i=2$ (resp. $i=3$)), hence must be equal to $\fq_i$. This implies that $\gamma_i$ factors through  $R^{\rm peu,\psi}_{\fq_i}\ra \widehat{R_{\fp_i}^{\rm ver,\psi}}$, and  it makes the diagram (\ref{equation-diagram})   commute using  the definitions of  (\ref{pspeu}), (\ref{ngf}). On the other hand, any morphism $\gamma':R^{\rm peu,\psi}_{\fq_i}\ra  \widehat{R_{\fp_i}^{\rm ver,\psi}}$ fitting into the commutative diagram (\ref{equation-diagram}) must be of the form above.

(ii) It can be proved similarly as in the proof of Proposition \ref{f1f2}.
The flatness of $\gamma_i$  follows from \cite[Theorem 23.1]{mat}, which applies as $R^{\rm peu,\psi}_{\fq_i}$ is regular and $\widehat{R_{\fp_i}^{\rm ver,\psi}}$ is Cohen-Macaulay, being the completion of a localization of the Cohen-Macaulay ring $R^{\rm ver,\psi}$. More concretely, one checks that \[\fm_{R^{\rm peu,\psi}_{\fq_i}}\widehat{R_{\fp_i}^{\rm ver,\psi}}=\fp_i\widehat{R_{\fp_i}^{\rm ver,\psi}}\] 
is   the maximal ideal, 
hence
$\dim \widehat{R_{\fp_i}^{\rm ver,\psi}}=\dim R^{\rm peu,\psi}_{\fq_i}+\dim \widehat{R_{\fp_i}^{\rm ver,\psi}}/\fm_{R^{\rm peu,\psi}_{\fq_i}}\widehat{R_{\fp_i}^{\rm ver,\psi}}=3$.

That $\gamma_i$ sends radical ideals to radical ideals follows from  \cite[Theorem 2.1]{Io}. Namely, it suffices  to show, say for $i=2$, that $\widehat{R_{\fp_2}^{\rm ver,\psi}}/\fm_{R^{\rm peu,\psi}_{\fq_2}}\widehat{R_{\fp_2}^{\rm ver,\psi}}\simeq \FF(\!( c_0, d_0)\!)$   is geometrically reduced over the residue field $\FF(\!(x_1)\!)$ of $R^{\rm peu,\psi}_{\fq_2}$, via the map $\gamma_i: x_1\mapsto d_0$; but  we have seen how to show this in the proof of Proposition \ref{f1f2}. The same argument goes through when $i=3$.
\end{proof}


\begin{remark}
One checks easily that there does not exist an $R^{\ps,\psi}$-homomorphism from $R^{\rm peu,\psi}$ to $R^{\ver,\psi}$.
\end{remark}

\section{The multiplicity of pseudo-deformation rings}
\label{section-pseudo-mult}

In this section, we will study the multiplicity of potentially semi-stable pseudo-deformation rings of $\overline{\rho}:=\ide\oplus\omega$.  \medskip

Recall that $R^{\ps,\psi}=R^{\ps,\psi}(\mathrm{tr}\rhobar)$ denotes the universal pseudo-deformation ring  of $\overline{\rho}$ with fixed determinant $\epsilon\psi$, where $\psi:G_{\Q_p}\ra \cO^{\times}$ is a continuous character. To lighten the notation, we will omit the superscript $\psi$ in the rest of the section; for example, we write $R^{\ps}$ for $R^{\ps,\psi}$. 

For   $\fn\in\mSpec R^{\ps}[1/p]$ we denote by $\kappa(\fn)$ the quotient field $R^{\rm ps}[1/p]/\fn$, $\cO_{\kappa(\fn)}$ the ring of integers of $\kappa(\fn)$,  and $T_{\fn}$ for the induced pseudo-deformation  of $\brho$ defined over $\kappa(\fn)$.

Denote by $I_{\irr}^{\ps}$ the intersection of all  $\fn\in\mSpec R^{\ps}[1/p]$ such that $T_{\fn}$ is the trace of an absolutely irreducible representation of $G_{\Qp}$ which is potentially semi-stable of type $(k,\tau,\psi)$, and by  $I_{1}^{\ps}$  (resp. $I_2^{\ps}$)
the intersection of all $\fn\in\mSpec R^{\ps}[1/p]$ such that $T_{\fn}$ is the trace of an absolutely reducible representation which is potentially semi-stable of type $(k,\tau,\psi)$ and contains a one-dimensional sub-representation lifting $\mathbbm{1}$ (resp. $\omega$) with the higher Hodge-Tate weight. We define in a similar way $I^{\ps}_{\rm cr, irr}$ and $I^{\ps}_{\mathrm{cr},i}$ ($i=1,2$) by replacing ``semi-stable'' by ``crystalline'' in the above definition. Here we note that for an indecomposable reducible  potentially semi-stable representation of distinct Hodge-Tate weights, the unique one-dimensional sub-representation is automatically of higher weight.


\begin{remark} \label{remark-I_{omega}}In the definition of $I^{\ps}_{i}$ (and $I^{\ps}_{\mathrm{cr}, i}$), we could have demanded that $T_{\fn}$ come from an \emph{indecomposable} reducible representation, because it follows from \cite{pr} that, for instance for $I^{\ps}_2$, if $\rho=\delta_1\oplus\delta_2$ is potentially semi-stable of type $(k,\tau,\psi)$, such that $\delta_1$ (resp. $\delta_2$) lifts $\ide$ (resp. $\omega$) and $G_{\Q_p}$ acts on $\delta_2$ via the higher Hodge-Tate weight, then any \emph{non-split} extension 
\[0\ra \delta_2\ra \rho'\ra \delta_1\ra0\] 
 is also potentially semi-stable of the same type.  Moreover, $\rho'$ is automatically potentially crystalline except when $k=2$ and $\tau=\chi\oplus\chi$ is scalar, in which case  $\delta_2\delta_1^{-1}=\epsilon$ and $\Ext^1_{G_{\Q_p}}(\delta_1,\delta_2)$ is 2-dimensional and we can always find a non-split potentially crystalline extension. 
\end{remark}

Fix a $p$-adic Hodge type $(k,\tau,\psi)$, and write  $V$ for $\sigma(k,\tau):=\Sym^{k-2}E^2\otimes \sigma(\tau)$ or $\sigma^{\rm cr}(k,\tau):=\Sym^{k-2}E^2\otimes\sigma^{\rm cr}(\tau)$ (when we consider potentially crystalline deformation rings). Choose  a $K$-stable $\cO$-lattice $\Theta$ inside $V$.  Let $N_1$, $N_2$ be respectively a projective envelope of $\pi_{\alpha}^{\vee}$ and of $\Sp^{\vee}$ in the category $\mathfrak{C}_{\psi}(\cO)$,  where $\pi_{\alpha}:=\Ind_P^G\alpha$ with $\alpha:=\omega\otimes\omega^{-1}$   the smooth character of $T:=\smatr{\Q_{p}^{\times}}00{\Q_p^{\times}}$. For $i=1,2$, set 
\[ M_i(\Theta):=\Hom_{\cO}(\Hom_{\cO\llbracket K\rrbracket}^{\rm cont}(N_i,\Theta^{d}),\cO), \]
where $\Theta^d$ denotes the Schikhof dual of $\Theta$ (see \cite{ST02}). 
Then $M_1(\Theta)$  and $M_2(\Theta)$ are natural compact $R^{\rm ps}$-modules where $R^{\ps}$ acts on $N_1$ and $N_2$ via the natural isomorphisms 
$R^{\ps}\cong \End_{\fC_{\psi}(\cO)}(N_1)\cong \End_{\fC_{\psi}(\cO)}(N_2)$  (cf. \cite[\S10]{pa10}).

\subsection{The module $M_1(\Theta)$}\label{subsection-M_1}

Recall that $\brho^1$ denotes a non-split extension of $\omega$ by $\ide$ (unique up to scalars),  $R^{\rm ver}({\brho^1})$ the universal deformation ring of $\brho^1$ with determinant $\epsilon\psi$ and $R^{\rm ver}(k,\tau,\brho^1)$ the potentially semi-stable deformation ring of type $(k,\tau,\psi)$. (The superscript $\psi$ is omitted as we remarked.)
The following theorem is a consequence of results of \cite{pa10},\cite{pa12}.

\begin{theorem}\label{theorem-M1}
We have an isomorphism $$\mathrm{Ann}(M_1(\Theta))\cong I^{\ps}_{\irr}\cap I^{\ps}_1$$  and an equality of 1-dimensional cycles (where $J$ is the prime ideal defined in \S\ref{subsection-J})
\[\mathcal{Z}_1\left(R^{\ps}/(I_{\irr}^{\ps}\cap I_1^{\ps},\varpi)\right)=a_{p-3,1}J.\]

The same statement holds if we replace $I^{\ps}_{\irr}$, $I^{\ps}_{1}$, $a_{p-3,1}$ by $I^{\ps}_{\rm cr,irr}$, $I^{\ps}_{\rm cr,1}$, $a_{p-3,1}^{\rm cr}$ respectively.
\end{theorem}
\begin{proof} 
Note that  $\check{\VV}(N_1)$ is isomorphic to the universal deformation of $\overline{\rho}^1$ by \cite[Corollary 10.72]{pa10}. 
By 
\cite[Corollary 6.5]{pa12} we know
\[R^{\rm ver}(\overline{\rho}^1)/\mathrm{Ann}_{R^{\ver}(\brho^1)}(M_1(\Theta))\cong R^{\rm ver}(k,\tau,\overline{\rho}^1).\]
The natural isomorphism $f^1: R^{\rm ps}\ra  R^{\rm ver}(\overline{\rho}^1)$ (see Proposition \ref{ngenericps}) induces an isomorphism from $R^{\rm ps}/(I_{\irr}^{\ps}\cap I_1^{\ps})$ to $R^{\rm ver}(k,\tau,\overline{\rho}^1)$. The first assertion  follows from this and the second assertion from \cite[Theorem 6.6]{pa12} and Proposition \ref{Prop-J-prime}, which say that
$\mathcal{Z}_1(R^{\ver}(k,\tau,\overline{\rho}^1)/\varpi)=a_{p-3,1}\mathcal{Z}_{1}(M_1(\sigma_{p-3,1}))=a_{p-3,1}J.$
\end{proof}

\subsection{The module $M_2(\Theta)$} 
\label{subsection-M_2}
We turn to study the action of $R^{\rm ps}$ on $M_2(\Theta)$.  
Recall that $N_2$ denotes a projective envelope of $\Sp^{\vee}$ in $\fC_{\psi}(\cO)$. For $\pi_1,\pi_2\in \mathrm{Mod}_{G,\psi}^{\rm l,fin}(\F)$ we will write $e^1(\pi_1,\pi_2):=\dim_{\F}\Ext^1_{G,\psi}(\pi_1,\pi_2)$. We refer to \cite[\S10.1]{pa10} for the list of   $e^1(\pi_1,\pi_2)$ when $\pi_1,\pi_2$ are both irreducible non-supersingular representations.

If $\mathrm{m}$ is an $R^{\ps}[1/p]$-module of finite length, we define as in \cite[\S2.2]{pa12}
\[\Pi(\mathrm{m}):=\Hom_{\cO}^{\rm cont}(N_2\otimes_{R^{\rm ps}}\mathrm{m}^0,E),\] where $\mathrm{m}^0$ is any $R^{\ps}$-stable $\cO$-lattice in $\mathrm{m}$ (the definition does not depend on the choice of $\mathrm{m}^0$). 
Equipped with the supremum norm, $\Pi(\mathrm{m})$ is an admissible unitary $E$-Banach space representation of $G$.

The following result is an analogue of \cite[Proposition 4.7]{pa12}. Recall from \cite{pa10} that an absolutely irreducible Banach space representation is called \emph{non-ordinary} if it is not a subquotient of a parabolic induction of a unitary character.
\begin{proposition}\label{Prop-Pi(k(n))}
For almost all $\fn\in\mSpec R^{\rm ps}[1/p]$, the $\kappa(\fn)$-Banach space representation $\Pi(\kappa(\fn))$ is either absolutely irreducible non-ordinary or fits into a non-split extension
\[0\ra (\Ind_{P}^G\delta_1\otimes \delta_2\epsilon^{-1})_{cont}\ra \Pi(\kappa(\fn))\ra (\Ind_P^G\delta_2\otimes\delta_1\epsilon^{-1})_{cont}\ra0,\] 
where $\delta_1,\delta_2:\Qp^{\times}\ra \kappa(\fn)^{\times}$ are unitary characters such that $\delta_1\delta_2=\epsilon\psi$ and $\delta_1\delta_2^{-1}\neq \mathbbm{1},\epsilon^{\pm1}$.
\end{proposition}

We need some preparations to prove this proposition. 
In the proof of the next lemma, we shall use Emerton's functor of ordinary parts defined in \cite{Em1}; our main reference for this is \cite[\S7.1]{pa10}. 

\begin{lemma}\label{Lemma-N_2-semi}
We have $\bigl((N_2\otimes_{R^{\rm ps}} \F)^{\vee}\bigr)^{\rm ss}\cong \Sp\oplus \ide^{\oplus2}\oplus \pi_{\alpha}^{\oplus2}$. 
\end{lemma}
\begin{proof}
First note that $N_2\otimes_{R^{\rm ps}}\F$ is the maximal quotient of $N_2$ which contains $\Sp^{\vee}$ with multiplicity one (in fact $\Sp^{\vee}$ must appear as its cosocle), or equivalently, $(N_2\otimes_{R^{\rm ps}}\F)^{\vee}$ is the (unique) maximal smooth $\F$-representation of $G$ with $G$-socle isomorphic to $\Sp$ and such that $(N_2\otimes_{R^{\rm ps}}\F)^{\vee}/\Sp$ contains no subquotient isomorphic to $\Sp$.  We now construct it explicitly. Consider the smooth $\F$-representation $\tau_1$ of $G$ defined in \cite[(181)]{pa10}, which fits into an exact sequence
\[0\ra \Sp\ra \tau_1\ra \ide\oplus\ide\ra0.\]
Moreover the $G$-socle of $\tau_1$ is $\Sp$.
By \cite[Lemma 10.12]{pa10},  $e^1(\pi_{\alpha},\tau_1)=2$, hence there exists an  extension of $\pi_{\alpha}\oplus\pi_{\alpha}$ by $\tau_1$, denoted by $\tau_1'$:
\begin{equation}\label{equation-tau1'}0\ra \tau_1\ra \tau_1'\ra \pi_{\alpha}\oplus\pi_{\alpha}\ra0,\end{equation}
such that the $G$-socle of $\tau_1'$ is still $\Sp$. In particular, we have an injection $\tau_1'\hookrightarrow (N_2\otimes_{R^{\rm ps}}\F)^{\vee}$. We shall prove that it is in fact an isomorphism. For this it suffices to show  $e^1(\pi,\tau_1')=0$ for all irreducible $\pi\in\mathrm{Mod}_{G,\psi}^{\rm sm}(\F)$ except when $\pi\cong \Sp$. Firstly, one checks $e^1(\ide,\tau_1')=0$, using that (see \cite[\S10.1]{pa10}) $$e^1(\ide,\ide)=0, \quad e^1(\ide,\pi_{\alpha})=0,\quad e^1(\ide,\Sp)=2.$$

We claim that $e^1(\pi_{\alpha},\tau_1')=0$. For this we need to use Emerton's functor of ordinary parts relative to $P$ (see \cite{Em1}).   We denote by $\Ord_P:\mathrm{Mod}^{\rm l,fin}_{G,\psi}(\F)\ra \mathrm{Mod}^{\rm l,fin}_{T,\psi}(\F)$ this functor and by $\R^i\Ord_P$ its right derived functors for $i\geq 1$.  It follows from \cite{EP} that $\R^i\Ord_P=0$ for $i\geq 2$.
Moreover we know by \cite[(182),(126)]{pa10} that  
\[\begin{array}{lclll}\mathrm{Ord}_P\tau_1=\ide,&\R^1\mathrm{Ord}_P\tau_1=(\alpha^{-1})^{\oplus2},\\ 
 \mathrm{Ord}_P\pi_{\alpha}=\alpha^{-1},& \R^1\mathrm{Ord}_P\pi_{\alpha}=\ide.\end{array}\]
Applying $\Ord_P$ to (\ref{equation-tau1'}) gives
 \[0\ra \ide\ra \Ord_P\tau_1'\ra (\alpha^{-1})^{\oplus2}\overset{\partial}{\ra} (\alpha^{-1})^{\oplus2}\ra \R^1\Ord_P\tau_1'\ra \ide^{\oplus2}\ra0.\]
 The connecting morphism $\partial$ must be injective. Indeed, if not, we would have that $\Hom_T(\Ord_P\tau_1',\alpha^{-1})\neq0$, hence $\Hom_{T}(\alpha^{-1},\Ord_P\tau_1')\neq0$ since there is no non-trivial $T$-extension between $\alpha^{-1}$ and $\ide$ (as $p\neq 2$). We then get 
 \[\Hom_G(\pi_{\alpha},\tau_1')\neq0\]
 by the adjointness property between $\Ord_P$ and $\Ind_P^G$ (see \cite[(120),(125)]{pa10}), which contradicts the definition of $\tau_1'$.
We  deduce that $\mathrm{Ord}_P\tau_1'\cong\ide$ and $\R^1\mathrm{Ord}_P\tau_1'\cong\ide^{\oplus2}$. Since $p\neq 2$, the claim follows from this and the exact sequence (see e.g. \cite[(123)]{pa10})
\[0\ra \Ext^1_{T,\psi}(\alpha,\mathrm{Ord}_P\tau_1')\ra \Ext^1_{G,\psi}(\pi_{\alpha},\tau_1')\ra \Hom_{T}(\alpha,\R^1\mathrm{Ord}_P\tau_1').\]

Since the block of $\Sp$ consists of $\{\Sp, \ide,\pi_{\alpha}\}$ by \cite[Proposition 5.42]{pa10}, we see that $\Ext^1_{G,\psi}(\pi,\tau_1')=0$ for all \emph{irreducible} $\pi\in \mathrm{Mod}_{G,\psi}^{\rm sm}(\F)$ except for $\pi\cong \Sp$. This shows that $(N_2\otimes_{R^{\rm ps}}\F)^{\vee}$ is isomorphic to $\tau_1'$, and the lemma follows.
 \end{proof}

Write $\mathfrak{B}$ for the block of $\Sp$, i.e. $\mathfrak{B}=\{\Sp,\ide,\pi_{\alpha}\}$.  Let $\mathrm{Ban}_{G,\psi}^{\rm adm,fl}(E)^{\mathfrak{B}}$ be the category of admissible unitary $E$-Banach space representations $\Pi$ of $G$, of finite length and with central character $\psi$,  such that all the irreducible constituents of $\overline{\Pi}^{\rm ss}$ lie in $\mathfrak{B}$. Here $\overline{\Pi}^{\rm ss}$ denotes the semi-simplification of the modulo $\varpi$ reduction of any open bounded $G$-invariant lattice in $\Pi$.
As in \cite[\S10]{pa10}, for $\fn$ a maximal ideal of $R^{\ps}[1/p]$, let $\mathrm{Ban}_{G,\psi}^{\rm adm,fl}(E)^{\mathfrak{B}}_{\fn}$ be the full subcategory of $\mathrm{Ban}_{G,\psi}^{\rm adm,fl}(E)^{\mathfrak{B}}$ consisting of those $\Pi$ such that $\mathrm{m}(\Pi)$ is killed by a power of $\fn$, where   $\mathrm{m}$ is defined as in \cite[Corollary 4.42]{pa10} with $\widetilde{P}=N_2$.
  
We will need to apply Colmez's functor $\check{\VV}$ to objects in $\mathrm{Ban}_{G,\psi}^{\rm adm,fl}(E)^{\mathfrak{B}}$. For such a $\Pi$, we define 
\[\check{\VV}(\Pi):=\check{\VV}(\Theta^d)\otimes_{\cO}E\] 
for any open bounded $G$-invariant $\cO$-lattice $\Theta$ in $\Pi$. Remark that $\check{\VV}$ is exact and \emph{contravariant} on $\mathrm{Ban}_{G,\psi}^{\rm adm,fl}(E)^{\mathfrak{B}}$. By the proof of \cite[Lemma 4.2]{pa12}, for $\mathrm{m}$ an $R^{\ps}[1/p]$-module of finite length, we have 
\begin{equation}\label{equation-isom-m}
\check{\VV}(\Pi(\mathrm{m}))\cong \check{\VV}(N_2)\otimes_{R^{\ps}}\mathrm{m}.\end{equation}
To see this, we just  tensor the sequence \cite[(22)]{pa12} with $E$ (over $\cO$).
 
\begin{lemma}\label{Lemma-Pi(kappa(fn))}
The representation $\Pi(\kappa(\fn))$ is nonzero, of finite length, and has an irreducible  $G$-socle (in the category $\mathrm{Ban}_{G,\psi}^{\rm adm,fl}(E)^{\mathfrak{B}}$).
In particular, it is indecomposable and lies in the category $\mathrm{Ban}_{G,\psi}^{\rm adm,fl}(E)_{\fn}^{\mathfrak{B}}$.
\end{lemma}
\begin{proof}
It follows from \cite[Lemma 4.25]{pa10}   that $\Pi(\kappa(\fn))$ is non-zero and of finite length. By Lemma \ref{Lemma-N_2-semi}, $N_2\otimes_{R^{\ps}}\F$ is of finite length in $\mathfrak{C}_{\psi}(\cO)$ and is finitely generated as an $\cO\llbracket K\rrbracket$-module, so  \cite[Corollary 4.33]{pa10} implies that $\Pi(\kappa(\fn))$ has an irreducible $G$-socle.   The last assertion follows from this and the decomposition of categories 
\[\mathrm{Ban}_{G,\psi}^{\rm adm,fl}(E)^{\mathfrak{B}}\cong \bigoplus_{\fn\in\mSpec R^{\rm ps}[1/p]}\mathrm{Ban}_{G,\psi}^{\rm adm,fl}(E)^{\mathfrak{B}}_{\fn}\]
established in  \cite[Corollary 10.106]{pa10}.
\end{proof}

Recall  that $R^{\rm ps}$ is isomorphic to $\cO\llbracket c_0,c_1,d_0,d_1\rrbracket /(c_0d_1-c_1d_0)$ by (\ref{equation-Rps-isom}). Let $\mathfrak{r}=(c_0,c_1)$ be the reducibility ideal of $R^{\rm ps}$. Also recall from \S4.1 that $N_1$ denotes a projective envelope of $\pi_{\alpha}^{\vee}$ in $\mathfrak{C}_{\psi}(\cO)$.

\begin{lemma}
  We have an exact sequence of $R^{\rm ps}[G_{\Q_p}]$-modules:
\begin{equation}\label{equation-remark-10.97}0\ra \mathfrak{r}.\check{\VV}(N_1)\ra \check{\VV}(N_2)\ra \rho^{\rm univ}_{\ide}\ra0,\end{equation}
where $\rho^{\rm univ}_{\ide}$ is the universal deformation of the trivial representation $\ide$ to $R^{\rm ps}/\mathfrak{r}$ (viewed as an $R^{\ps}$-module).
\end{lemma}
\begin{proof}
This follows from \cite[Remark 10.97]{pa10}. In fact it gives a commutative diagram of $R^{\rm ps}[G_{\Q_p}]$-modules:
\[\xymatrix{0\ar[r]&\check{\VV}(N_2)\ar[r]\ar[d]&\check{\VV}(N_1) \ar[r]\ar[d]&\rho_{\omega}^{\rm univ}\ar[r]\ar@{=}[d]&0\\
0\ar[r]&\rho^{\rm univ}_{\ide}\ar[r]&\check{\VV}(N_1)/\mathfrak{r}.\check{\VV}(N_1)\ar[r]&\rho^{\rm univ}_{\omega}\ar[r]&0 }\]
and the result follows from the snake lemma. 
\end{proof}

\begin{lemma}\label{Lemma-r-kappa(n)}
Assume $\fn$ contains the reducibility ideal $\mathfrak{r}$. Then $\mathfrak{r}\otimes_{R^{\rm ps}}{\kappa(\fn)}$ is of dimension 2 over ${\kappa(\fn)}$ if $\fn= (c_0,c_1,d_0,d_1)$ and of  dimension 1 otherwise.
\end{lemma}
\begin{proof}
Write $f=c_0d_1-c_1d_0$ so that $R^{\rm ps}\cong \cO\llbracket c_0,c_1,d_0,d_1\rrbracket /(f)$. Let $\mathfrak{n}_0:=\mathfrak{n}\cap R^{\rm ps}$ so that $R^{\rm ps}/\mathfrak{n}_0\cong\cO_{\kappa(\fn)}$ and
\[\mathfrak{r}\otimes_{R^{\rm ps}}\kappa(\fn)\cong \mathfrak{r}/(\mathfrak{r}.\fn_0)\otimes_{\cO_{\kappa(\fn)}}\kappa(\fn).\] 
In particular if $\fn=(c_0,c_1,d_0,d_1)$, we have $f\in \mathfrak{r}.\fn_0$ and see easily that $ \mathfrak{r}/\mathfrak{r}.\fn_0$ is free of rank 2 over $\cO\cong\cO_{\kappa(\fn)}$, generated by $c_0,c_1$.    When $\fn\neq (c_0,c_1,d_0,d_1)$, then making a base change from $\kappa(\fn)$ to a finite field  extension $\kappa'$,  we may assume that $\fn=(c_0,c_1,d_0-t_0',d_1-t_1')$ with $t_i'\in \kappa'$ and at least one of them is non-zero, say $t_0'\neq 0$.  This implies that ${c}_1={c}_0t_1't_0'^{-1}$ in $\mathfrak{r}\otimes_{R^{\ps}}\kappa'$, hence the latter $\kappa'$-space is of dimension 1 (it is nonzero by Nakayama's lemma).  The lemma follows. 
\end{proof}

\begin{proof}[Proof of Proposition \ref{Prop-Pi(k(n))}]
Suppose first that $T_{\fn}$  is absolutely irreducible. By  \cite[Proposition 10.107(i)]{pa10}, the category $\mathrm{Ban}_{G,\psi}^{\rm adm,fl}(E)^{\mathfrak{B}}_{\fn}$ contains only one absolutely irreducible object denoted by $\Pi_{\fn}$, which must be non-ordinary. In particular, each irreducible subquotient of $\Pi(\kappa(\fn))$ is isomorphic to $\Pi_{\fn}$ and Lemma \ref{Lemma-Pi(kappa(fn))} gives an injection $\Pi_{\fn}\hookrightarrow \Pi(\kappa(\fn))$. Lemma \ref{Lemma-N_2-semi} implies that the setup of \cite[Proposition 4.32]{pa10} is satisfied, which implies that $\mathrm{m}(\Pi(\kappa(\fn))/\Pi_{\fn})=0$ (we use the notation $\mathrm{m}$ as in \emph{loc.cit.}), hence $\Pi(\kappa(\fn))/\Pi_{\fn}=0$.

Suppose from  now on that $T_{\fn}$ is absolutely reducible and can be written as $T_{\fn}=\delta_1+\delta_2$ over a finite extension $L'$ of $\kappa(\fn)$ with $\delta_1\delta_2^{-1}\neq \ide,\epsilon^{\pm1}$. Since $\delta_1\neq \delta_2$ as they reduce to different characters, we have only excluded the case when $T_{\fn}=\delta+\delta\epsilon$ (with $\delta^2=\psi$). Using the isomorphism $R^{\ps}\cong R^{\ver}(\overline{\rho}^1)$, \cite[Corollary 10.94]{pa10} implies that we only exclude the ideal $(c_0,c_1,d_0,d_1)$.  

We first treat the case when $L'=\kappa(\fn)$.  Up to order, we may assume that $\delta_1$ reduces to $\ide$ modulo the maximal ideal of $\cO_{\kappa(\fn)}$, and therefore $\delta_2$ reduces to $\omega$. Then \cite[Proposition 10.107(ii)]{pa10} implies that $\mathrm{Ban}_{G,\psi}^{\rm adm,fl}(E)^{\mathfrak{B}}_{\fn}$ has exactly two (non-isomorphic) absolutely irreducible objects $\Pi_1$ and $\Pi_2$, where
\[\Pi_1=(\Ind_{P}^G\delta_1\otimes\delta_2\epsilon^{-1})_{cont}, \quad\Pi_2=(\Ind_P^G\delta_2\otimes\delta_1\epsilon^{-1})_{cont}.\]
Let $\Pi$ be the unique irreducible Banach space sub-representation of $\Pi(\kappa(\fn))$ given by Lemma \ref{Lemma-Pi(kappa(fn))}.  Since  $\overline{\Pi}^{\rm ss}$ contains $\Sp$ as a subquotient, we have $\Pi\cong \Pi_1$ by our convention. Moreover, by the assumption $p\geq 5$  we must have  
$\overline{\Pi_2}^{\rm ss}\cong \pi_{\alpha}$.
Put \[\Pi':=\Pi(\kappa(\fn))/\Pi.\] 
As in the irreducible case, \cite[Proposition 4.32]{pa10} implies that   each irreducible subquotient of $\Pi'$ is isomorphic to $\Pi_2$.  To conclude we need to show that $\Pi\cong \Pi_2$, or equivalently  $\check{\VV}(\Pi_2)\cong \delta_2$.

Tensoring the sequence (\ref{equation-remark-10.97}) with ${\kappa(\fn)}$ (over $R^{\rm ps}$) gives
\[ \mathfrak{r}.\check{\VV}(N_1)\otimes_{R^{\rm ps}}{\kappa(\fn)}\overset{\phi}{\ra} \check{\VV}(N_2)\otimes_{R^{\rm ps}}{\kappa(\fn)}\ra \rho^{\rm univ}_{\ide}\otimes_{R^{\rm ps}}{\kappa(\fn)}\ra0.\]
 On the one hand, since $ \fn$ contains the reducibility ideal $\mathfrak{r}$, $\rho_{\ide}^{\rm univ}\otimes_{R^{\ps}}\kappa(\fn)$ is non-zero and $\rho_{\ide}^{\rm univ}\otimes_{R^{\ps}}\cO_{\kappa(\fn)}$ is a deformation of $\ide$ to $\cO_{\kappa(\fn)}$. By our convention, this implies that $\rho_{\ide}^{\rm univ}\otimes_{R^{\ps}}\kappa(\fn)$ is isomorphic to $\delta_1$.  On the other hand, since  $\check{\VV}(N_1)$ is the universal deformation of $\overline{\rho}^1$ over $R^{\rm ver}(\brho^1)\simeq R^{\rm ps}$, it is flat over $R^{\rm ps}$. Together with Lemma \ref{Lemma-r-kappa(n)}, this implies that \[\mathfrak{r}.\check{\VV}(N_1)\otimes_{R^{\rm ps}}\kappa(\fn)\cong \check{\VV}(N_1)\otimes_{R^{\rm ps}}(\mathfrak{r}\otimes_{R^{\rm ps}}\kappa(\fn))\cong \check{\VV}(N_1)\otimes_{R^{\ps}}\kappa(\fn).\]
which is isomorphic to a non-split extension of $\delta_2$ by $\delta_1$ by \cite[Proposition 4.9(ii)]{pa12}.  The map $\phi$ can not be injective, since $\check{\VV}(N_2)\otimes_{R^{\ps}}\kappa(\fn)$ does not contain $\delta_1$ as a \emph{sub}-representation (otherwise, $\Pi(\kappa(\fn))$ would admit $\Pi_1$ as a \emph{quotient} which contradicts Lemma \ref{Lemma-Pi(kappa(fn))}). As consequence, $\mathrm{Im}(\phi)\cong \delta_2$, and $\check{\VV}(N_2)\otimes_{R^{\ps}}\kappa(\fn)$ is a non-split extension of $\delta_1$ by $\delta_2$.

For general $L'$, the same argument as above shows that $\check{\VV}(\Pi(\kappa(\fn)))\otimes_{\kappa(\fn)}L'$, which is isomorphic to $\check{\VV}(N_2)\otimes_{R^{\ps}}L'$ by (\ref{equation-isom-m}), is a non-split extension of $\delta_1$ by $\delta_2$. Since $\ide\neq \omega$ (as $p>2$),  \cite[Lemma 4.5]{pa12} implies that $\delta_1, \delta_2$ are in fact defined over $\kappa(\fn)$. As in the proof of \cite[Proposition 4.9]{pa12}, we see that $\Pi(\kappa(\fn))$ is a non-split extension of $\Pi_2$ by $\Pi_1$.   
\end{proof}

\begin{remark}
We thank Pa\v{s}k\={u}nas for pointing out to us that $N_2$ is not flat over $R^{\rm ps}$.
\end{remark}

\begin{proposition}\label{conditionc}
If $V=\sigma(k,\tau)$ (resp. $V=\sigma^{\mathrm{cr}}(k,\tau)$), then  $$\dim_{\kappa(\fn)}\Hom_K(V,\Pi(\kappa(\fn)))\leq 1$$ for almost all $\fn\in \mSpec R^{\rm ps}[1/p]$. Moreover,  for such $\fn$, $\dim_{\kappa(\fn)}\Hom_K(V,\Pi(\kappa(\fn)))=1$ if and only if $T_{\fn}$ is absolutely irreducible and potentially semi-stable (resp. potentially crystalline)  of type $(k,\tau,\psi)$, or $T_{\fn}$ is reducible and isomorphic to the trace of a potentially semi-stable  (resp. potentially crystalline) representation of type $(k,\tau,\psi)$ which is non-split and contains a one-dimensional sub-representation  lifting $\omega$. 
\end{proposition}

\begin{proof}
We exclude the finite set of $\fn$ as in Proposition \ref{Prop-Pi(k(n))}. The case when $T_{\fn}$ is absolutely irreducible is identical to that of  \cite[Proposition 4.14]{pa12}. Assume that $T_{\fn}$ is absolutely reducible. Then by the proof of Proposition \ref{Prop-Pi(k(n))}, $T_{\fn}$ can be written of the form $\delta_1+\delta_2$ over $\kappa(\fn)$ with $\delta_1\delta_2^{-1}\neq \ide,\epsilon^{\pm1}$, and $\Pi(\kappa(\fn))$ fits into a non-split extension
\[0\ra \Pi_1\ra \Pi(\kappa(\fn))\ra\Pi_2\ra0\] 
with $\Pi_1,\Pi_2$ absolutely irreducible and non-isomorphic. As in the proof of Proposition \ref{Prop-Pi(k(n))}, we assume that  $\delta_1$ reduces to $\ide$ and $\delta_2$ reduces to $\omega$  modulo the maximal ideal of $\cO_{\kappa(\fn)}$, so that $\check{\VV}(\Pi_i)\cong\delta_i$ for $i=1,2$. Now the proof of \cite[Proposition 4.14]{pa12} gives that $\Pi(\kappa(\fn))^{\rm alg}$, the subspace of locally algebraic vectors in $\Pi(\kappa(\fn))$, is non-zero if and only if $\Pi_1^{\rm alg}$ is non-zero, if and only if the $G_{\Q_p}$-representation   \begin{equation}\label{equation-VV(Pi)-nonsplit}0\ra \delta_2\ra\check{\VV}(\Pi(\kappa(\fn)))\ra\delta_1\ra0\end{equation} is potentially semi-stable (resp. potentially crystalline if $V=\sigma^{\rm cr}(k,\tau)$) of type $(k,\tau,\psi)$. We conclude as in the proof of \emph{loc.cit.}, noting that the sequence (\ref{equation-VV(Pi)-nonsplit}) is non-split since  $\Pi(\kappa(\fn))$ is a non-split extension of $\Pi_2$ by $\Pi_1$.
\end{proof}

Recall the fixed $K$-stable lattice $\Theta$ in $V$ and the  $R^{\ps}$-module  $M_2(\Theta)$.  As in \S4.1, we have the following result.
\begin{theorem}\label{Theorem-M2}
We have an isomorphism 
\[\mathrm{Ann}(M_2(\Theta))\cong I_{\irr}^{\ps}\cap I_{2}^{\ps}\]
and an equality of $1$-dimensional cycles 
\[\cZ_1\left(R^{\rm ps}/(I_{\irr}^{\ps}\cap I_{2}^{\ps},\varpi)\right)=(a_{0,0}+a_{p-1,0})J.\]
 
 The same statement holds if we replace $I^{\rm ps}_{\irr}$, $I^{\rm ps}_{2}$, $a_{0,0}$, $a_{p-1,0}$ by $I^{\rm ps}_{\rm cr,irr}$, $I^{\rm ps}_{\rm cr,2}$, $a_{0,0}^{\rm cr}$, $a_{p-1,0}^{\rm cr}$ respectively.
\end{theorem}
\begin{proof}
Write $\Sigma$ for the set of $\fn$ in the statement of Proposition \ref{conditionc} such that $\dim_{\kappa(\fn)}\Hom_K(V,\Pi(\kappa(\fn)))=1$. By  Proposition \ref{conditionc} and Remark \ref{remark-I_{omega}}, we see that $\Sigma$ forms a dense subset of $\Spec \left(R^{\ps}/(I_{\irr}^{\ps}\cap I_{2}^{\ps})\right)[1/p]$, hence of $\Spec R^{\ps}/(I_{\irr}^{\ps}\cap I_{2}^{\ps})$, see \cite[Remark 2.43]{pa12}.
Now \cite[Proposition 2.22]{pa12} implies that $\Sigma$ forms a dense subset of the support of $M_2(\Theta)$, so we get the equality $\sqrt{\mathrm{Ann}(M_2(\Theta))}=I^{\ps}_{\irr}\cap I_2^{\ps}$.

To prove the theorem, we need to check the conditions (a),(b),(c) in \cite[Theorem 2.42]{pa12} in order  to apply it. The condition (a)  follows from the definition of $N_2$, using the main result of \cite{EP}.  The condition (c)(i) is just Proposition \ref{conditionc} and (c)(ii) proceeds exactly as in \cite[\S4.2]{pa12} using the main result of \cite{Do} and Proposition \ref{conditionc} in place of \cite[Proposition 4.9]{pa12}. 

We are left to verify the condition (b). By \cite[Proposition 2.29]{pa12}, it suffices to prove that $M_2(\Theta)$ is a finitely generated Cohen-Macaulay $R^{\ps}$-module. Recall that we have constructed an element $x\in R^{\ps}$ in the proof of Proposition \ref{Prop-J-prime}, which is a lifting of $S$ via the surjection $R^{\ps}\twoheadrightarrow \F\llbracket S\rrbracket$. We claim that $(\varpi,x)$ forms a regular sequence for $M_2(\Theta)$. Firstly, since $N_2$ is projective in $\mathrm{Mod}_{K,\zeta}^{\rm pro}(\cO)$, the exact sequence $0\ra \Theta\overset{\varpi}{\ra} \Theta\ra \Theta/\varpi\Theta\ra0$  induces an exact sequence of $R^{\ps}$-modules
\[0\ra M_2(\Theta)\overset{\varpi}{\ra} M_2(\Theta)\ra M_2(\Theta/\varpi\Theta)\ra0.\]
This implies that $\varpi$ is regular for $M_2(\Theta)$ and $M_2(\Theta)/\varpi M_2(\Theta)\cong M_2(\Theta/\varpi\Theta)$.  Secondly, it follows from the exact sequence (\ref{equation-S-regular}) (in Proposition \emph{loc.cit.}) that $x$ is regular for $M_2(\sigma)$ for any smooth irreducible $\F$-representation $\sigma$ of $K$, hence also regular for $M_2(\Theta/\varpi\Theta)$ (here we use that $N_2$ is projective in $\mathrm{Mod}_{K,\zeta}^{\rm pro}(\cO)$). Moreover,  the quotient $M_2(\Theta/\varpi\Theta)/xM_2(\Theta/\varpi\Theta)$ is of Krull dimensional $0$ since this is true for $M_2(\sigma)/xM_2(\sigma)$ by (\ref{equation-S-regular}). This proves the claim.  Finally, Lemma \ref{Lemma-N_2-semi} and \cite[Proposition 2.15]{pa12} imply that $M_2(\Theta)$ is finitely generated over $R^{\ps}$.

All conditions of  \cite[Theorem 2.42]{pa12} being verified, we deduce that $\mathrm{Ann}(M_2(\Theta))$ is a radical ideal, hence the equality $\mathrm{Ann}(M_2(\Theta))=I^{\ps}_{\irr}\cap I_2^{\ps}$. We also deduce an equality of $1$-dimensional cycles  
\[\cZ_1\left(R^{\rm ps}/(I_{\irr}^{\ps}\cap I_{2}^{\ps},\varpi)\right)=\sum_{n,m}a_{n,m}\cZ_1(M_2(\sigma_{n,m})).\]
But it follows from Proposition \ref{Prop-J-prime} that $M_2(\sigma_{n,m})\neq 0$ if and only if $(n,m)=(0,0)$ or $(p-1,0)$, in which case the associated 1-dimensional cycle is $J$. 
\end{proof}

\section{Proof of the Breuil-M\'{e}zard conjecture}\label{proofbm}
In this section we prove the Breuil-M\'ezard conjecture for the residual representation $ \ide\oplus \omega$. To do this, we study the relation between potentially semi-stable pseudo-deformation rings and potentially semi-stable (of the same type) deformation rings so that we can use what is proved in Section 3 and Section 4 to deduce the multiplicities of potentially semi-stable deformation rings (modulo $\varpi$). 
\medskip
 
\emph{Notational remark}: As the character $\psi$ will be fixed everywhere, we omit it from the notation of the deformation rings for simplicity. For $\m\in\mSpec R^{\ver}[1/p]$, write $\rho_{\fm}$ for the associated deformation of $\brho$.
\medskip

Let  $\overline{\rho}$ be an extension of two distinct characters $\chi_2$ by $\chi_1$ and fix a $p$-adic Hodge type $(k,\tau,\psi)$. A closed point in $\Spec R^{\ver}(k,\tau,\overline{\rho})[1/p]$ is called of reducibility type $\rm irr$ if  the corresponding $G_{\Qp}$-representation  is absolutely irreducible.
For a closed point $x\in  \Spec R^{\ver}(k,\tau,\overline{\rho})[1/p]$ such that the corresponding $G_{\Qp}$-representation $V_x$   is reducible,  it has to be an (possibly split) extension of two distinct characters $\delta_i$ lifting  $\chi_i$, respectively. We say the point $x$ is of reducibility type $\chi_i$, or more briefly, of type $i$,  if $\delta_i$ has the higher Hodge-Tate weight.

For  $*\in\{\irr,1,2\}$, define an ideal  $I^{\ver}_*$ of $R^{\ver}$ as follows:
\[ I_{*}^{\rm ver}:=\Bigl(\bigcap_{\fm\in\mSpec R^{\rm ver}[1/p]}\fm\Bigr)\cap R^{\rm ver},\]for $\fm$ ranging over all the maximal ideals such that $\rho_{\fm}$ is potentially semi-stable of type $(k,\tau,\psi)$ and of reducibility type $*$, so that
\[R^{\ver}(k,\tau,\brho)= R^{\ver}/(I^{\ver}_{\irr}\cap I^{\ver}_{1}\cap I^{\ver}_2).\] 
In the pseudo-deformation ring $R^{\ps}=R^{\ps}(\mathrm{tr}\brho)$,  define the ideal
\[ I_{*}^{\rm ps}:=I^{\ver}_*\cap R^{\rm ps}.\] One sees that this definition coincides with the one defined at the beginning of Section 4.

We define in an obvious way the ideals $I^{\ver}_{\rm cr,*}$ and $I^{\ps}_{\rm cr,*}$ ($*\in\{\irr,1,2\}$) by considering potentially crystalline representations of type $*$.

\begin{remark}

In \cite{ki09}, a quotient ring $R^{\ps}_{U_0}$ of $R^{\ps}$ (denoted by $R^{\ps}(k,\tau,\brho)$ in \cite{bm2}) is introduced, which can be seen as the analogue of $R^{\ver}(k,\tau,\brho)$.  One checks that $R^{\ps}_{U_0}=R^{\ps}/(I^{\ps}_{\irr}\cap I^{\ps}_1\cap I^{\ps}_2)$ and that $I^{\ps}_*$  defines the (closure of union of) components  in $\Spec R^{\ps}_{U_0}[1/p]$  of type $*$ defined  loc.\ cit.

\end{remark}

In the rest of this section, we will take $\overline{\rho}= \mathbbm{1}\oplus\omega$ and use the convention $\chi_1=\mathbbm{1}$ and $\chi_2=\omega$ while we talk about reducibility types.

Recall from \S \ref{maps} that there are three minimal prime ideals of $R^{\ver}$ containing $JR^{\rm ver}$: $\fp_1,\fp_2,\fp_3$, and that $JR^{\ver}=\fp_1\cap\fp_2\cap\fp_3$. 
We first record the following fact, which says that they induce all possible minimal prime ideals of $R^{\rm ver}(k,\tau,\brho)/\varpi$. 

\begin{proposition}\label{prop-Iver-3}
(i) The quotient ring $R^{\rm ver}/(I^{\rm ver}_{\irr}+(\varpi))$ has at most three minimal prime ideals, that is among $\{\fp_1,\fp_2,\fp_3\}$. 

(ii) The quotient ring $R^{\rm ver}/(I^{\ver}_{1}+(\varpi))$ has at most one minimal prime ideal  $\fp_1$, with the quantity being  one if and only if $I^{\ver}_1\neq R^{\ver}$. 

(iii) The quotient ring $R^{\rm ver}/(I^{\ver}_2+(\varpi))$ has at most two minimal prime ideals $\fp_2$ and $\fp_3$.

The same hold in the crystalline case, i.e. with $I_*^{\ver}$ replaced by $I_{\rm cr, *}^{\ver}$.
\end{proposition}
\begin{proof}
(i) Let $*\in\{\irr,1,2\}$ and $\fq\in\Spec R^{\ver}$ be any minimal prime ideal over $I^{\rm ver}_*+(\varpi)$. By Theorems \ref{theorem-M1} and \ref{Theorem-M2}, $J$ is the radical of the ideal $I^{\ps}_*+(\varpi)$.  Since the natural map $f^{\rm ver}:R^{\rm ps}\ra R^{\rm ver}$ maps $I^{\rm ps}_*$ into $I^{\rm ver}_*$, there exists  $r\in \N$ large enough such that
\[(\fp_1\cap \fp_2\cap \fp_3)^rR^{\rm ver}=J^rR^{\rm ver}\subset (I^{\ps}_*+(\varpi))R^{\rm ver}\subset  \fq.\]
Hence $\fq$ must contain one of the $\fp_1,\fp_2,\fp_3$.  
By Theorem \ref{pst} and \cite[Theorem 31.5]{mat}, $R^{\rm ver}/(I^{\ver}_{*}+(\varpi))$ is equidimensional of dimension $2$, which implies that $\fq$ has height 3. Since $\fp_i$ ($i=1,2,3$) also has height 3, the first claim follows. 
 
The proofs of (ii) and (iii) are similar and we only give that of (ii). We follow the arguments in the proof of \cite[Lemma 4.3.4(ii)]{bm2}. Let $\fq'$ be a minimal prime ideal over $I^{\ver}_{1}$. By the proof of (i) we only need to show $\fq'\nsubseteq \fp_2,\fp_3$. As in the proof \emph{loc.cit.}, the associated deformation 
\[\rho_{\fq'}:G_{\Q_p}\ra \GL_2(R^{\ver}/\fq')\]  
is reducible and it contains a free sub-$R^{\ver}/\fq'$-module  of rank $1$ as a direct summand, which is a deformation of the trivial character $\mathbbm{1}$.  The same property holds for any prime ideal of $R^{\ver}$ containing $\fq'$. However, by the explicit description of $\rho_{R^{\ver}}$ in \S \ref{splitng}, the deformations $\rho_{{\fp_2}}$ and $\rho_{{\fp_3}}$  are reducible \emph{non-split}, containing a free sub-module of rank $1$ lifting $\omega$. This implies $\fq'\nsubseteq \fp_2,\fp_3$ and the result follows.
\end{proof}

By \cite[Theorem 14.7]{mat} we have
\begin{equation}\label{sumof3}e(R^{\rm ver}(k,\tau,\brho)/\varpi)=\sum_{i=1}^3\ell(R^{\ver}(k,\tau,\brho)_{\fp_i}/\varpi)e(R^{\ver}/\fp_i)=\sum_{i=1}^3\ell(R^{\ver}(k,\tau,\brho)_{\fp_i}/\varpi)\end{equation}
where the second equality holds because $e(R^{\ver}/\fp_i)=1$ for $i=1,2,3$. We are left to study $\ell(R^{\ver}(k,\tau,\brho)_{\fp_i}/\varpi)$, which is also equal to $e(R^{\ver}(k,\tau,\brho)_{\fp_i}/\varpi)$. Of course, the same happens in the crystalline case.
    
\subsection{Multiplicities at $\fp_1$ and $\fp_2$}  \label{subsection-5.1}

Recall the maps (\ref{ngf1})   $f_i^{\ver}: R^{\rm ps}\ra R_{\fp_i}^{\rm ver}$, for   $i=1,2,3$.     
  
\begin{proposition}\label{Prop-Rps-Rver-II}
 (i) 
For $i=1,2$, we have $I_{\irr}^{\rm ver}R^{\ver}_{\fp_i}=I^{\ps}_{\irr}R^{\ver}_{\fp_i}$ and $I_{i}^{\rm ver}R_{\fp_i}^{\rm ver}=I^{\ps}_{i}R^{\ver}_{\fp_i}$.

 (ii)  For $*\in\{\irr,2\}$, we have  $I^{\ver}_{*}R^{\ver}_{\fp_3}[1/p]=\sqrt{I^{\ps}_{*}R^{\ver}_{\fp_3}[1/p]}$. 
\end{proposition}

\begin{proof}
 
 First look at $I^{\ver}_{\irr}$. Using the fact that  $R^{\rm ver}[1/p]$ is a Jacobson ring, we have by definition
\[ I_{\irr}^{\rm ver}R^{\ver}[1/p]=\bigcap_{\fn\in\mSpec R^{\rm ver}[1/p]}\fn, \quad \sqrt{ I_{\irr}^{\rm ps}R^{\ver}[1/p]}= \bigcap_{\m\in\mSpec R^{\rm ver}[1/p]}\m,\]
where $\fn$ ranges over all maximal ideals such that $\rho_{\fn}$ is absolutely irreducible of type $(k,\tau,\psi)$, and $\fm$ ranges over all maximal ideals such that $\mathrm{tr}(\rho_{\fm})$ is absolutely irreducible of type $(k,\tau,\psi)$, that is $\mathrm{tr}(\rho_{\fm})\cong\mathrm{tr}(\rho')$ for some $\rho'$ which is absolutely irreducible of type $(k,\tau,\psi)$.  Clearly these conditions define the same subset of $\mSpec R^{\ver}[1/p]$, hence the equality
\begin{equation}\label{equation-equality-1/p}
I^{\ver}_{\irr}R^{\ver}[1/p]=\sqrt{I^{\ps}_{\irr}R^{\ver}[1/p]}=\sqrt{I^{\ps}_{\irr}R^{\ver}}[1/p]
\end{equation}where the second equality holds because taking radical commutes with localization. 
Taking localization at $\fp_i$ (viewing the two sides as $R^{\ver}$-modules), $i=1,2,3$, gives
\begin{equation}\label{equation-I-irr}
I^{\ver}_{\irr}R^{\ver}_{\fp_i}[1/p]=\sqrt{I^{\ps}_{\irr}R^{\ver}_{\fp_i}}[1/p]\end{equation}
hence (ii) holds for $*=\rm irr$. To deduce (i), first remark that if $A$ is an $\cO$-algebra and $I$ is an ideal of $A$ such that the quotient $A/I$ is an $\cO$-flat module, then $I=(IA[1/p])\cap A$.  
Since the map $R^{\ps}\ra R^{\ver}_{\fp_i}$ (here $i=1,2$) is flat by Proposition \ref{f1f2},   $R_{\fp_i}^{\ver}/I^{\ps}_{\irr}R_{\fp_i}^{\ver}$ is $\cO$-flat as $R^{\ps}/I^{\ps}_{\irr}$ is. This implies that $R^{\ver}_{\fp_i}/\sqrt{I^{\ps}_{\irr}R^{\ver}_{\fp_i}}$ is also $\cO$-flat and thus (\ref{equation-I-irr}) improves to be \[I^{\ver}_{\irr}R^{\rm ver}_{\fp_i}=\sqrt{I^{\ps}_{\irr}R^{\ver}_{\fp_i}},\quad i=1,2.\] 
Then we conclude still by Proposition \ref{f1f2}, which says that 
 $I^{\ps}_{\irr}R^{\ver}_{\fp_i}$ is already radical.  So far we have proved (i) and (ii) for $*=\rm irr$.

The claim for $I^{\ver}_i$ ($i=1,2$) is proved similarly, using Proposition \ref{prop-Iver-3}.   More precisely, with the notation in the proof of \emph{loc. cit.} let $\fn\in \Spec((R^{\ver}/\fq')[1/p])$ be any closed point such that $\tr\rho_{\fn}$ comes from some potentially semi-stable representation of type $(k,\tau,\psi)$. Since we have fixed its reducibility type, the representation $\rho_{\fn}$ itself has to be potentially semi-stable of type $(k,\tau,\psi)$.   The  rest of the proof  then goes over as in the irreducible case.
\end{proof}

\begin{remark}
In general, we do not expect  $I^{\ps}_{*}R^{\ver}_{\fp_3}=I^{\ver}_{*}R^{\ver}_{\fp_3}$ to be true (this would imply $I^{\ps}_{*}R^{\ver}=I^{\ver}_{*}$). For example, in the crystalline case it could happen that $I^{\ps}_{\rm cr,\irr}=(c_0-p,c_1,d_1)$ in $R^{\rm ps}$. 
Then $I^{\ps}_{\rm cr,\irr}R^{\ver}=(c_0d_1-c_1d_0,bc_0-p,bc_1,d_1)$ and $R^{\ver}/I^{\ps}_{\rm cr,\irr}R^{\ver}$ has $\fp_3$ as a minimal prime ideal, which implies that $R^{\ver}/I^{\ps}_{\rm cr,\irr}R^{\ver}$ is not equidimensional, while $R^{\ver}/I^{\ver}_{\rm cr,\irr}$ is equidimensional by Theorem \ref{pst}. 
\end{remark}

\begin{proposition}\label{prop-mult-fp12}
For $i=1,2$, we have 
 
\[\ell\left(R^{\ver}(k,\tau,\brho)_{\fp_i}/\varpi\right)=\ell\left(R^{\ps}_J/(I^{\ps}_{\irr}\cap I^{\ps}_i,\varpi)\right).\] 
\end{proposition} 
\begin{proof} 
It follows from Proposition \ref{prop-Iver-3} that $R^{\ver}(k,\tau,\brho)_{\fp_i}\cong R^{\ver}_{\fp_i}/(I^{\ver}_{\irr}\cap I^{\ver}_{i})$ for $i=1,2$.
Then Proposition \ref{Prop-Rps-Rver-II}(i) implies further that  \[R^{\ver}(k,\tau,\brho)_{\fp_i}\cong R^{\ps}_J/(I^{\ps}_{\irr}\cap I_i^{\ps})\otimes_{R^{\ps}_J}R^{\ver}_{\fp_i}.\]
Since the local map $R^{\ps}_J\ra R^{\ver}_{\fp_i}$ is flat by (the proof of) Proposition \ref{f1f2}, so is  $R^{\ps}_J/(I^{\ps}_{\irr}\cap I^{\ps}_i,\varpi)\ra R^{\ver}(k,\tau,\brho)_{\fp_i}/\varpi$. Applying Lemma \ref{Lemma-Kisin-flat} below to it we obtain \[\ell\left(R^{\ver}(k,\tau,\brho)_{\fp_i}/\varpi\right)=\ell(R^{\ps}_J/(I^{\ps}_{\irr}\cap I^{\ps}_i,\varpi))e(R^{\ver}_{\fp_i}/J)=\ell(R^{\ps}_J/(I^{\ps}_{\irr}\cap I^{\ps}_i,\varpi)).\] Here we have used  the fact that $e(R^{\ver}_{\fp_i}/J)=1$.
\end{proof}

 \begin{lemma}\label{Lemma-Kisin-flat}
Let $A\ra B$ be a local map of Noetherian local rings with radicals $\fm$ and $\fn$, respectively. Let $\fp\subset A$ be a nilpotent prime ideal and suppose that all the minimal prime ideals of $B$ lie over $\fp$. Assume further that $B$ is flat over $A$.  Then
\[e_{\fn}(B)=e_{\fn/\fp B}(B/\fp B)\ell(A_{\fp}).\]
\end{lemma}
\begin{proof}
Let $\{\fq_1,\cdots,\fq_m\}$ be the set of minimal prime ideals of $B$. By \cite[Theorem 14.7]{mat}, 
we have
 \[e_{\fn}(B)=\sum_{i=1}^me_{\fn/\fq_i}(B/\fq_i)\ell_{B_{\fq_i}}(B_{\fq_i})\]
 and 
\[e_{\fn/\fp B}(B/\fp B)=\sum_{i=1}^me_{\fn/\fq_i}(B/\fq_i)\ell_{(B/\fp B)_{\fq_i}}((B/\fp B)_{\fq_i}).\]
Since $A\ra B$ is flat, so is $A_{\fp}\ra B_{\fq_i}$ for any $i$. By Nagata's flatness theorem (see for example \cite[Ex. 22.1]{mat}), we have
\[\ell_{B_{\fq_i}}(B_{\fq_i})=\ell_{A_{\fp}}(A_{\fp})\cdot \ell_{B_{\fq_i}}(B_{\fq_i}/\fp B_{\fq_i}).\]
The result follows.

Note that we can also adapt the proof of \cite[1.3.10]{ki09}, where all the inequalities appeared become equalities under the assumption that $B$ is flat over $A$. 
\end{proof}

\begin{remark}\label{bmgenericsplit}
In this remark, we take $\brho$ to be of the form $\brho\cong\chi_1\oplus\chi_2$ with $\chi_1\chi_2^{-1}\notin\{\ide,\omega^{\pm1}\}$. The situation is simpler, in the sense that the analogue of Proposition \ref{prop-Iver-3} holds except that the minimal ideal $\fp_3$ disappears. In this case, there are only two minimal prime ideals of $R^{\ver}$ containing $JR^{\rm ver}$; in the notation of Remark \ref{flatgeneric}, $J=(\varpi,y_2,y_3)$. By Remark \ref{flatgeneric}, the natural homomorphism $R^{\ps}\ra R^{\ver}$ is \emph{flat} and maps radical ideals to radical ideals. If we let $\brho^1$ (reap. $\brho^2$) be the unique non-split extension of $\chi_2$ by $\chi_1$ (resp. of $\chi_1$ by $\chi_2$), then we have 
 \[e(R^{\ver}(k,\tau,\brho)/\varpi)=e(R^{\ver}(k,\tau,\brho^1)/\varpi)+e(R^{\ver}(k,\tau,\brho^2)/\varpi).\]
which proves the Breuil-M\'ezard conjecture in this case; the conjecture for the two terms on the right hand side are already known by \cite{ki09} and \cite{pa12}. The crystalline case is shown in the same way.

\end{remark}

\subsection{Multiplicity at $\fp_2$  and $\fp_3$} \label{subsection-5.2}
 
 We determine the multiplicity of $R^{\ver}(k,\tau,\brho)/\varpi$ at $\fp_2$ and $\fp_3$, by means of deformation rings of peu ramifi\'e extensions, for which the Breuil-M\'{e}zard conjecture has been treated in \cite{pa12}.
\medskip

  Recall the   map (\ref{pspeu}) 
 \[f^{\rm peu}: R^{\rm ps}\simeq \frac{\cO\llbracket c_0,c_1,d_0,d_1\rrbracket}{(c_0d_1-c_1d_0)}\hookrightarrow R^{\rm peu}\simeq \cO\llbracket x_1,x_2,x_3\rrbracket, \]
\[c_0\mapsto x_3, \quad c_1\mapsto x_2x_3, \quad d_0\mapsto x_1,\quad d_1\mapsto x_1x_2.\]
Here $R^{\rm peu}:=R^{\ver}(\brho^{\rm peu})$ denotes the universal deformation ring (with fixed determinant $\epsilon\psi$) of $\brho^{\rm peu}$, the (non-split) peu ramifi\'e extension of $\ide$ by $\omega$. Recall that $R^{\rm peu}/JR^{\rm peu}$ has two minimal prime ideals $\fq_2=(\varpi,x_2,x_3)$ and $\fq_3=(\varpi,x_1,x_3)$.

By Proposition \ref{f3}  we have the following commutative diagram (\ref{equation-diagram})
\begin{equation*}
\xymatrix{R^{\rm ps}\ar_{f^{\rm peu}}[d]\ar^{{f^{\ver}}}[dr] \\
 R^{\rm peu}_{\fq_i}\ar[r]^{\gamma_i}&\widehat{R_{\fp_i}^{\rm ver}}}\end{equation*}
In the proof of Proposition \ref{f3}, we have seen that $\fp_i\widehat{R^{\ver}_{\fp_i}}$ lies over $\fq_iR^{\rm peu}_{\fq_i}$ ($i=2,3$) and $\fq_i\widehat{R^{\ver}_{\fp_i}}=\fp_i\widehat{R^{\ver}_{\fp_i}}$,  under the map $\gamma_i$ (\ref{peus}).  

 Denote by $I_{\irr}^{\rm peu}$ (resp. $I^{\rm peu}_{2}$) the ideal of $R^{\rm peu}$ cutting out the closure in $\Spec R^{\rm peu} $ of closed points in $\Spec R^{\rm ver}(k,\tau,\brho^{\rm peu})[1/p]$ which are of irreducible type (resp. of reducible type). The notation $I^{\rm peu}_2$ is chosen as a component of reducible type is automatically of type 2.

\begin{proposition}\label{peuver}
We have for $i=2,3$ the following relations under the map $\gamma_i$ (\ref{peus}): 
\[I^{\rm peu}_{\irr}\widehat{R^{\rm ver}_{\fp_i}}=I^{\rm ver}_{\irr}\widehat{R^{\ver}_{\fp_i}},\quad  I^{\rm peu}_{2}\widehat{R^{\rm ver}_{\fp_i}}=I^{\rm ver}_{2}\widehat{R^{\ver}_{\fp_i}}.\] \end{proposition}
\begin{proof}
By Proposition \ref{Prop-Rps-Rver-II}, we have for $*\in\{\irr,2\}$ and $i\in\{2,3\}$
\[I^{\ver}_{*}R_{\fp_i}^{\ver}[1/p]=\sqrt{I^{\ps}_{*}R_{\fp_i}^{\ver}[1/p]}=\sqrt{I^{\ps}_{*}R^{\ver}_{\fp_i}}[1/p].\] 
  Applying Lemma \ref{Lemma-localisation-completion} below to $A=R^{\ver}_{\fp_i}$ and $I=\sqrt{I^{\rm ps}_{*}R^{\ver}_{\fp_i}}$, $J=I^{\ver}_{*}R^{\ver}_{\fp_i}$ we get   
\begin{equation}\label{equation-Rver}I^{\ver}_{*}\widehat{R^{\ver}_{\fp_i}}[1/p]=\sqrt{I^{\ps}_{*}R^{\ver}_{\fp_i}}\widehat{R^{\ver}_{\fp_i}}[1/p]=\sqrt{I^{\rm ps}_{*}\widehat{R^{\ver}_{\fp_i}}}[1/p],\end{equation}
where to get the second equality we have applied Lemma \ref{Lemma-localisation-completion}(ii) to $A=R^{\ver}_{\fp_i}$ which is a Nagata ring, being a localization of a complete noetherian local ring (see \cite[Chapitre IX, \S4, n$^{\circ}$4]{bour}).
On the other hand, a similar proof  as in Proposition \ref{Prop-Rps-Rver-II}  shows  
\begin{equation}\label{equation-Rpeu}I^{\rm peu}_{*}R^{\rm peu}_{\fq_i}[1/p]=\sqrt{I^{\ps}_*R^{\rm peu}_{\fq_i}[1/p]}=\sqrt{I^{\ps}_*R^{\rm peu}_{\fq_i}}[1/p].\end{equation}
 Then using the commutative diagram (\ref{equation-diagram}), we get
\[I^{\ver}_{*}\widehat{R^{\ver}_{\fp_i}}[1/p]\overset{(\ref{equation-Rver})}{=}\sqrt{f^{\ver}(I_*^{\ps})\widehat{R^{\ver}_{\fp_i}}}[1/p]=\sqrt{I^{\ps}_*R^{\rm peu}_{\fq_i}}\widehat{R^{\ver}_{\fp_i}}[1/p]\overset{(\ref{equation-Rpeu})}{=} I^{\rm peu}_{*}\widehat{R^{\ver}_{\fp_i}}[1/p].\]
Here, we use (the proof of) Lemma \ref{Lemma-localisation-completion}(ii), applied to the morphism $\gamma_i$, to  get the second equality, since  $\gamma_i$ sends radical ideals to radical ideals by Proposition \ref{f3}. 
Since  $\gamma_i: R^{\rm peu}_{\fq_i}\ra \widehat{R^{\ver}_{\fp_i}}$  is flat by Proposition \ref{f3} again, we conclude as in the proof of Proposition \ref{Prop-Rps-Rver-II}(i).
\end{proof} 

\begin{lemma}\label{Lemma-localisation-completion}
Let $(A,\m)$ be a noetherian  local  ring and denote by $\hat{A}$ its $\m$-adic completion.

(i)  Let $I\subseteq J$ be two ideals of $A$ such that   $IA[1/p]=JA[1/p]$. Then we  have $I\hat{A}[1/p]=J\hat{A}[1/p]$. 

(ii) If moreover $A$ is a Nagata ring, then the natural morphism $A\ra \hat{A}$ sends radical ideals to radical ideals. In particular, $\sqrt{I\hat{A}}=\sqrt{I}\hat{A}$ for any ideal $I$ of $A$.
\end{lemma}
\begin{proof}
(i) Write $M=I/J$ and consider the exact sequence of $A$-modules:
\begin{equation}\label{equation-I-J}
0\ra I\ra J\ra M\ra0.\end{equation}
The assumption that $IA[1/p]=JA[1/p]$ implies that $M[1/p]=0$. Since $M$ is a finitely generated $A$-module, we can find $n\in\N$ large enough such that $p^nm=0$ for all $m\in M$. Taking $\m$-adic completions and inverting $p$, the sequence (\ref{equation-I-J}) induces an exact  sequence
\[0\ra I\hat{A}[1/p]\ra J\hat{A}[1/p]\ra \hat{M}[1/p]\ra0.\]
By definition we have $\hat{M}=\varprojlim_{i\geq 1}M/\m^iM$, so that $\hat{M}$ is also killed by $p^n$ and therefore $\hat{M}[1/p]=0$. The result follows. 

(ii) By Nagata-Zariski theorem, see for example \cite[Theorem 1.3]{Io}, the natural morphism $A\ra \hat{A}$ is a reduced morphism, hence sends radical ideals to radical ideals. To show the last assertion, we remark that  for any ring morphism $f: A\ra B$ which sends radical ideals to radical ideals and any ideal $I$ of $A$, we have $\sqrt{I}B=\sqrt{IB}$. Indeed,   the inclusion $\subseteq$ holds in general, and the  inclusion $\supseteq$ holds because $IB\subseteq \sqrt{I}B$ and $\sqrt{I}B$ is already radical. 
\end{proof}

\begin{corollary}\label{peu23}
We have the equality
\[\ell(R^{\ver}(k,\tau,\brho)_{\fp_2}/\varpi)+\ell(R^{\ver}(k,\tau,\brho)_{\fp_3}/\varpi)=a_{0,0}+2a_{p-1,0}.\]
\end{corollary}
\begin{proof}
To lighten the notation, denote $R^{\rm peu}(k,\tau):=R^{\ver}(k,\tau,\brho^{\rm peu})$. First, a similar proof as that of Proposition \ref{prop-Iver-3} implies that $R^{\rm peu}(k,\tau)/\varpi$ has at most 2 minimal prime ideals $\fq_2$ and $\fq_3$, so that by \cite[Theorem 14.7]{mat}
\[e(R^{\rm peu}(k,\tau)/\varpi)=\ell(R^{\rm peu}(k,\tau)_{\fq_2}/\varpi)+\ell(R^{\rm peu}(k,\tau)_{\fq_3}/\varpi),\]where we have used that $R^{\rm peu}/\fq_2$ and $R^{\rm peu}/\fq_3$ both have Hilbert-Samuel multiplicity $1$.
Since we know $e(R^{\rm peu}(k,\tau)/\varpi)=a_{0,0}+2a_{p-1,0}$ by the Breuil-M\'ezard conjecture  for $\brho^{\rm peu}$ which is proved in \cite{pa12}, it suffices to show 
\begin{equation}\label{equation-fp-fq}\ell(R^{\ver}(k,\tau,\brho)_{\fp_i}/\varpi)=\ell(R^{\rm peu}(k,\tau)_{\fq_i}/\varpi),\ \ \ i=2,3.\end{equation}

Proposition \ref{prop-Iver-3}   and Proposition \ref{peuver} imply that  
  \[ \widehat{R^{\ver}(k,\tau,\brho)_{\fp_i}}\cong R^{\rm peu}(k,\tau)_{\fq_i}\otimes_{R^{\rm peu}_{\fq_i}} \widehat{R^{\ver}_{\fp_i}}.\]
Note that taking completion does not change Hilbert-Samuel multiplicities.  
Then using that   $\fq_i\widehat{R^{\ver}_{\fp_i}}=\fp_i\widehat{R^{\ver}_{\fp_i}}$ for $i=2,3$, we get (\ref{equation-fp-fq}) by applying Lemma \ref{Lemma-Kisin-flat} to the flat map $R^{\rm peu}(k,\tau)_{\fq_i}/\varpi \ra \widehat{R^{\ver}(k,\tau,\brho)_{\fp_i}}/\varpi$, base change of the flat local morphism $R^{\rm peu}_{\fq_i}\ra \widehat{R^{\ver}_{\fp_i}}$, as in the proof of Proposition \ref{prop-mult-fp12}.
\end{proof}

\subsection{Conclusion}
We can now prove the  (cycle version of)  Breuil-M\'ezard conjecture for $\brho=\ide\oplus\omega$. First we prove it for potentially semi-stable deformation rings.

  \begin{theorem}\label{bmng}

The cycle version of the Breuil-M\'{e}zard Conjecture (hence the original Conjecture \ref{bm}) is true for the representation  $\overline{\rho}=\ide\oplus\omega$. Precisely, we have
\[\cZ(R^{\ver}(k,\tau,\brho)/\varpi)=a_{p-3,1}\fp_1+a_{0,0}\fp_2+a_{p-1,0}(\fp_2+\fp_3).\]

\end{theorem}

\begin{proof}
 Theorem \ref{theorem-M1}, Theorem \ref{Theorem-M2} and Proposition \ref{prop-mult-fp12} imply that  
\[\ell(R^{\rm ver}(k,\tau,\brho)_{\fp_1}/\varpi)=a_{p-3,1},\quad \ell(R^{\rm ver}(k,\tau,\brho)_{\fp_2}/\varpi)=a_{0,0}+a_{p-1,0}.\]
Together with Corollary \ref{peu23}, this implies that $\ell(R^{\rm ver}(k,\tau,\brho)_{\fp_3}/\varpi)=a_{p-1,0}$.  They prove the theorem by (\ref{sumof3}).
\end{proof}

To prove the  Breuil-M\'{e}zard Conjecture for potentially crystalline deformation rings, it is enough to assume that the Galois type $\tau$ is scalar, since otherwise potentially semi-stable and potentially crystalline deformation rings coincide by  \cite[Lemma 2.2.2.2]{bm1}.  
  
\begin{theorem}\label{bmngcris}
The cycle version of the crystalline Breuil-M\'{e}zard Conjecture (hence the original Conjecture \ref{bm}) holds for $\overline{\rho}=\ide\oplus\omega$:
\[\cZ(R_{\rm cr}^{\ver}(k,\tau,\brho)/\varpi)=a_{p-3,1}^{\rm cr}\fp_1+a_{0,0}^{\rm cr}\fp_2+a_{p-1,0}^{\rm cr}(\fp_2+\fp_3).\]
\end{theorem}
 
  \begin{proof}

In the case $k>2$, all the previous arguments in \S\S\ref{subsection-5.1}-\ref{subsection-5.2} go over verbatim with $I^{\ps}_*$, $I^{\ver}_*$ and $R^{\ver}(k,\tau,\brho)$ replaced by $I^{\ps}_{\rm cr,*}$, $I^{\ver}_{\rm cr,*}$ and $R_{\rm cr}^{\ver}(k,\tau,\brho)$, respectively. For example, Proposition \ref{Prop-Rps-Rver-II}, which is the key result, holds true,  since  a representation is potentially crystalline  of type $(k,\tau,\psi)$ if and only if its trace is.

We are left to treat the special case $k=2$. In this case there are crystalline representations and semi-stable non-crystalline representations with the same trace, which makes Proposition \ref{Prop-Rps-Rver-II}(ii)  fail when $*=2$. However, we give a direct proof in this case. After twisting, we may assume $\tau=\ide$ is the trivial type and $\psi$ is the trivial character. 

 First of all, Theorem \ref{bmng} implies  that    $$\cZ(R^{\ver}(2,\ide,\brho)/\varpi)=\fp_2+\fp_3$$ since $\overline{\sigma(2,\ide)}^{\rm ss}=\sigma_{p-1,0}$. By definition, $\Spec R^{\ver}_{\rm cr}(2,\ide,\brho)/\varpi$ is a union of irreducible components of $\Spec R^{\ver}(2,\ide,\brho)/\varpi$.
Moreover, we know by  \cite[Proposition 3.6]{kw2} that $R_{\rm cr}^{\ver}(2,\ide,\brho)$ is formally smooth, which implies that the cycle $\mathcal{Z}(R^{\ver}_{\rm cr}(2,\ide,\brho)/\varpi)$ is simply of the form $\fp_i$ for some $i\in\{2,3\}$.  However, we cannot have $i=3$, since the image of $\Spec (R^{\ver}/\fp_3)$ in $\Spec R^{\ps}$  reduces to  the closed point, whereas   that of $\Spec R^{\ver}_{\rm cr}(2,\ide,\brho)/\varpi$ does not because we can find easily two crystalline liftings of $\brho$ with distinct traces.  Hence we have   $\cZ(R_{\rm cr}^{\ver}(k,\tau,\brho)/\varpi)=\fp_2$ which proves the theorem since $\overline{\sigma^{\rm cr}(2,\ide)}^{\rm ss}=\sigma_{0,0}$.
\end{proof}

  \section{The Fontaine-Mazur conjecture}\label{section6}
 
This section is devoted to the proof of Theorem \ref{fm}. Since the arguments for deducing the Fontaine-Mazur conjecture from the Breuil-M\'ezard conjecture  are now standard, thanks to  \cite{ki09} (and its errata in \cite{GK}), we only emphasize how to modify Kisin's original proof in the cases that are not covered in \cite{ki09}.  In the following, whenever we quote a result in \cite[\S 2]{ki09}, we mean the corrected version given in \cite[Appendix B]{GK}. 
\medskip
 
Let $F$ be a totally real field in which $p$ is split. Let $D$ be a quaternion algebra with centre $F$, ramified at all infinite places and a set of finite places $\Sigma$ which does not contain the places above $p$. Let $U\subset (D\otimes_F\mathbb{A}_F^f)^{\times}$ be the open compact as in \cite[2.1.1]{ki09}. Fix a continuous representation $\sigma: U\ra \Aut(\prod_vW_{\sigma_v})$ such that  \[W_{\sigma_v}=\mathrm{Sym}^{k_v-2}\cO_{F_v}^2\otimes\sigma(\tau_v)\otimes \mathrm{det}^{w_v}, \quad  \forall v|p\] with $w_v$ an integer and $\tau_v: I_v\ra \GL_2(E)$ a representation with open kernel, and $\sigma$ is trivial at other places. Fix a character $\psi: (\mathbb{A}_F^f)^{\times}/F^{\times}\ra \cO^{\times}$ so that at any $U_v\cap \cO_{F_v}^{\times}$, $\sigma$ is given by $\psi$. 
Extend $\sigma$ to be a representation of the product $U(\mathbb{A}_F^f)^{\times}$ by letting the second component act by $\psi$. Let $S_{\sigma,\psi}(U,\cO)$ be the set of continuous function  $f: D^{\times}\backslash(D\otimes_F\mathbb{A}_{F}^f)^{\times}\ra \prod_vW_{\sigma_v}$ defined in \cite[2.1.1]{ki09}, which is chosen  to be a finite projective  $\cO$-module by shrinking $U$; cf. \cite[2.1.2]{ki09}. 

We  take $S$ to  be the union of $\Sigma_p:=\Sigma\cup \{v,v|p\}$ and some other unramified places  $v$ such that $U_v\subset D_v^{\times}$ consists of matrices which are upper triangular and unipotent modulo $\varpi_v$.  Consider a continuous absolutely irreducible representation \[\overline{\rho}: G_{F,S}\ra \GL_2(\FF)\] such that there is an eigenform $f\in S_{\sigma,\psi}(U,\cO)$ with the associated Galois representation reducing to $\overline{\rho}$; cf. \cite[2.2.3]{ki09} for additional technical conditions on $\rhobar$. We have the universal deformation ring $R_{F,S}:=R^{\rm ver}(\overline{\rho})$ analogous to the local setting.

In the following, it is more convenient to use  the universal framed deformation rings; see, for example, \cite[Section 2]{kw2} for basics. Note that by \cite[Proposition 2.1]{kw2} a universal framed deformation ring $R^{\square}$ is formally smooth over a corresponding versal deformation ring $R^{\ver}$, and that all the closed points of $\Spec R^{\square}[1/p]$ lying above a given closed point of $\Spec R^{\rm ver}[1/p]$ give rise to isomorphic representations. Hence our main results in Section \ref{proofbm} hold for framed deformation rings. 

We add the superscript $\square$ to the notation of  deformation rings to indicate framed deformations, and as before use the superscript $\psi$ to indicate the deformations with fixed determinant $\psi$. Among them, the universal framed deformation ring $R_{F,S}^{\square}$ of the global absolutely irreducible $\overline{\rho}$ is defined by considering deformations of $\overline{\rho}$, together with the lifts of  a fixed basis of (the representation space of) $\overline{\rho}|_{G_{F_v}}$ for each $v\in \Sigma_p$.  In particular, this gives a natural map of $\cO$-algebras $R_{\Sigma_p}^{\square,\psi}:=\widehat{\otimes}R_v^{\square,\psi}\ra R_{F,S}^{\square,\psi}$, where $R_v^{\square,\psi}$ is the local framed deformation ring of $\rhobar|_{G_{F_v}}$. We denote the various quotient rings  analogously. 
 
 Let $Q_n$ (for any $n\geq 1$) be the set of auxiliary primes as in \cite[2.2.4]{ki09}, for which $h:=|Q_n|=\mathrm{dim}_{\FF}H^1(G_{F,S},\mathrm{ad}^0\overline{\rho}(1))$ is independent of $n$ and $R_{S_{Q_n}}^{\square,\psi}$ ($S_{Q_n}=S\cup Q_n$ and  $U_v$ for $v\in Q_n$ are defined as in \cite[2.1.6]{ki09}) is topologically generated by $g=h+j-d$ elements  as an $R_{\Sigma_p}^{\square,\psi}$-algebra, with
$j=4|\Sigma_p|-1$ and $d=[F:\mathbb{Q}]+3|\Sigma_p|$. Set \[M_n=S_{\sigma,\psi}(U_{Q_n},\cO)_{\fm_{Q_n}}\otimes_{R_{F,S_{Q_n}}^{\psi}}R_{F,S_{Q_n}}^{\square,\psi},\] where the ideal $\fm_{Q_n}$ is associated to $\overline{\rho}$ and $Q_n$ as in \cite[2.1.5, 2.1.6]{ki09}, and $U_{Q_n}=\prod_{v\in Q_n}U_v$.

Fix a $K$-stable filtration of $W_{\sigma}\otimes_{\cO}\F$ by $\FF$-vector spaces:
\[0=L_0\subset\cdots\subset L_s=W_{\sigma}\otimes_{\cO}\FF,\] such that the graded piece $\sigma_i=L_{i+1}/L_i$ is absolutely irreducible, which then has the form $\sigma_i=\otimes_{v|p}\sigma_{n_{i,v},m_{i,v}}$, with $n_{i,v}\in \{0,\cdots,p-1\}$ and $m_{i,v}\in \{0,\cdots,p-2\}$.  This induces a filtration $\{M_n^i\}$ on $M_n\otimes_{\cO}\FF$ for any $n\geq 0$. Let $\mathfrak{c}_n\subset \cO\llbracket y_1,\cdots, y_{h+j}\rrbracket$ be the ideal as in \cite[2.2.9]{ki09}. There are maps of $R_{\infty}=R_{\Sigma_p}^{\square,\psi}\llbracket x_1,\cdots,x_g\rrbracket$-modules $f_n: M_{n+1}/\mathfrak{c}_{n+1}M_{n+1}\ra M_n/\mathfrak{c}_nM_n$ compatible with the filtrations (modulo $\varpi$). The $R_{\infty}$-module $M_{\infty}=\varprojlim M_n/\mathfrak{c}_nM_n$ is finite free as an $\cO\llbracket y_1,\cdots,y_{h+j}\rrbracket$-module, whose reduction mod $\varpi$ has a filtration \[0=M_{\infty}^0\subset\cdots\subset M_{\infty}^s=M_{\infty}\otimes_{\cO}\FF,\] each of whose graded pieces is a finite free $\FF\llbracket y_1,\cdots,y_{h+j}\rrbracket$-module.

 As explained in \cite[2.2.10]{ki09}, the action of $R_{v}^{\square,\psi}$ on $M_{\infty}$ for  $v|p$   factors through the potentially semi-stable quotient $\bar{R}_{v}^{\square,\psi}$, twist of $R^{\square,\psi}(k_v,\tau_v,\overline{\rho}|_{G_{F_v}}\otimes \omega^{-w_v})$, and for $v\in \Sigma$ factors through certain quotient $\bar{R}_v^{\square,\psi}$ whose closed points parametrize extensions of $\gamma_v$ by $\gamma_v(1)$, where    $\gamma_v$ is the unramified character such that $\gamma_v^2=\psi|_{G_{F_v}}$. Denote $\bar{R}_{\Sigma_p}^{\square,\psi}:=\widehat{\otimes}_{v\in \Sigma_p}\bar{R}_v^{\square,\psi}$. It can be shown that  $\bar{R}_{\Sigma_p}^{\square,\psi}$ is of relative dimension $d$ over $\cO$.  Now $M_{\infty}$ is an $\bar{R}_{\infty}=\bar{R}_{\Sigma_p}^{\square,\psi}\llbracket x_1,\cdots,x_g\rrbracket$-module.

Let $i\in \{1,\cdots,s\}$. For $v\in \Sigma$ and $v|p$ such that  $\overline{\rho}|_{G_{F_v}}$ is not a twist of $
\left(
\begin{array}{ccc}
  \omega & *     \\
  0&\mathbbm{1}   
\end{array}
\right)$, let $\bar{R}_{v,i}^{\square,\psi}$ be as   in the proof of \cite[2.2.15]{ki09}.  Otherwise, we define  for $v|p$ with $\overline{\rho}|_{G_{F_v}}$ (possibly split) peu ramifi\'{e} (resp. tr\`{e}s ramifi\'{e})  that  \[\bar{R}_{v,i}^{\square,\psi}=R^{\square,\psi_{i,v}}(2,(\tilde{\omega}^{m_{i,v}})^{\oplus 2},\overline{\rho}|_{G_{F_v}})/\varpi_v\] with  $\psi_{i,v}:G_{F_v}\ra \cO^{\times}$ any character such that  $\psi_{i,v}|_{I_{F_v}}=\epsilon^{n_{i,v}}\tilde{\omega}^{2m_{i,v}}$ and $\psi_{i,v}\equiv \psi|_{G_{F_v}}$ mod $\varpi_v$ (cf. \cite[2.2.13]{ki09}). That is, we use the semi-stable instead of crystalline deformation rings in the latter cases as building blocks, because of the appearance of components of semi-stable non-crystalline points. Then we form the completed tensor product $\bar{R}_{\Sigma_p,i}^{\square,\psi}$ of the $\bar{R}_{v,i}^{\square,\psi}$ for all $v\in \Sigma_p$ and set $\bar{R}_{\infty}^i:=\bar{R}_{\Sigma_p,i}^{\square,\psi}\llbracket x_1,\cdots,x_g\rrbracket$.


\begin{lemma}\label{prescribe}

For any $i=1,\cdots,s$, the support of the  $\bar{R}_{\infty}^i$-module $M_{\infty}^i/M_{\infty}^{i-1}$ is all of  $\Spec \bar{R}_{\infty}^i$.

\end{lemma}

\begin{proof}

This is a modification of the proof of \cite[2.2.15]{ki09}, which uses the existence of modular liftings of prescribed type. For the latter in the cases that $\overline{\rho}|_{G_{F_v}}$ is  a twist of $
\left(
\begin{array}{ccc}
  \omega & *     \\
  0&\mathbbm{1}   
\end{array}
\right)$  ($v|p$), which is not treated in \cite{ki09}, we use \cite[Theorem 9.7]{kw2} as follows. 

Suppose we are  in these cases.  
By \cite[Theorem 5.3.1(i)]{bm1}, we know that the cycle $\cZ(R^{\square,\psi_{i,v}}(2,(\tilde{\omega}^{m_{i,v}})^{\oplus 2},\overline{\rho}|_{G_{F_v}})/\varpi)$ is irreducible if $\overline{\rho}$ is tr\`{e}s ramifi\'{e}.  In the (possibly split) peu ramifi\'{e} case, it is the sum of two irreducible components, one of which is just $\cZ(R_{\rm cr}^{\square,\psi_{i,v}}(2,(\tilde{\omega}^{m_{i,v}})^{\oplus 2},\overline{\rho}|_{G_{F_v}})/\varpi)$, and the other of which is the closure of the semi-stable non-crystalline points, as predicted by the Breuil-M\'{e}zard conjecture. 
Now \cite[Theorem 9.7]{kw2} tells us that the support of $M_{\infty}^i/M_{\infty}^{i-1}$, as an $\bar{R}_{\infty}^i$-module, meets each irreducible component of $\bar{R}_{\infty}^i$, and in fact consists of all of it by dimension counting; cf. the proof of \cite[2.2.15]{ki09}.
\end{proof}

  \begin{proposition}\label{faithful}

  $M_{\infty}$ is a faithful $\bar{R}_{\infty}$-module. 
 \end{proposition}

\begin{proof}

Recall Theorem \ref{bmng} and the main result of \cite{pa12}.  Now the result follows from Lemma \ref{prescribe} and   the argument of \cite[2.2.17]{ki09}.
\end{proof}

 \begin{theorem} Let $F$ be a totally real field in which $p$ splits. Let $\rho: G_{F,S}\ra \GL_2(\cO)$ be a continuous representation such that $\rhobar$ is odd, $\overline{\rho}|_{G_{F(\zeta_p)}}$ is absolutely irreducible, the restriction $\rho|_{G_{F_v}}$ for each place $v|p$  is potentially semi-stable of distinct Hodge-Tate weights, and the residual representation $\overline{\rho}$ is modular. Then $\rho$ comes from a Hilbert modular form. 
 
 As a consequence, Theorem \ref{fm} holds. 
 \end{theorem}

\begin{proof}
By Proposition \ref{faithful} and \cite[2.2.11]{ki09}, the modularity holds  in the case that $\rho|_{I_{F_v}}, v\in \Sigma$, is an extension of $\gamma_v$ by $\gamma_v(1)$.  The general case then follows from the base change arguments as in the proof of \cite[2.2.18]{ki09}. For Theorem \ref{fm}, one only needs that $\overline{\rho}$ is modular, which  is the main result of \cite{kw1}, \cite{kw2}.
\end{proof}

\vspace{\baselineskip}

Yongquan Hu

IRMAR - UMR CNRS 6625

Campus Beaulieu, 35042 Rennes cedex, France

E-mail: yongquan.hu@univ-rennes1.fr

\vspace{\baselineskip}
 \vspace{\baselineskip}

Fucheng Tan

Department of Mathematics, Michigan State University

East Lansing, Michigan 48824, USA

E-mail: ftan@math.msu.edu

\end{document}